\documentclass[reqno,11pt]{amsart}
\usepackage[letterpaper, margin=1in]{geometry}
\RequirePackage{amsmath,amssymb,amsthm,graphicx,mathrsfs,url,slashed,subcaption,empheq}
\RequirePackage[usenames,dvipsnames]{xcolor}
\RequirePackage[colorlinks=true,linkcolor=Red,citecolor=Green]{hyperref}
\RequirePackage{amsxtra}
\usepackage{cancel}
\usepackage{tikz, quiver, wrapfig}

\makeatletter
\@namedef{subjclassname@2020}{\textup{2020} Mathematics Subject Classification}
\makeatother

\title{An $\infty$-Laplacian for differential forms, and calibrated laminations}
\author{Aidan Backus}
\address{Department of Mathematics, Brown University}
\email{aidan\_backus@brown.edu}
\date{\today}
\keywords{laminations, $\infty$-Laplacian, convex duality, minimal hypersurfaces, calibrations, functions of least gradient}
\subjclass[2020]{primary: 35J94; secondary: 35J92, 53C38, 49N15, 49Q20}

\newcommand{\RR}{\mathbf{R}}

\newcommand{\Hyp}{\mathbf H}
\newcommand{\Sph}{\mathbf S}

\newcommand{\Ball}{\mathbf{B}}

\newcommand*\dif{\mathop{}\!\mathrm{d}}

\DeclareMathOperator{\id}{id}
\DeclareMathOperator{\Hom}{Hom}
\DeclareMathOperator{\PD}{PD}
\DeclareMathOperator{\coker}{coker}
\DeclareMathOperator{\supp}{supp}
\DeclareMathOperator{\sech}{sech}

\newcommand{\weakto}{\rightharpoonup}

\newcommand{\normal}{\mathbf n}

\newcommand{\vol}{\mathrm{vol}}

\newcommand{\Lip}{\mathrm{Lip}}

\newcommand{\Riem}{\mathrm{Riem}}

\newcommand{\Mass}{\mathbf M}
\newcommand{\Comass}{\mathbf L}

\newcommand{\dfn}[1]{\emph{#1}\index{#1}}

\newcommand{\loc}{\mathrm{loc}}
\newcommand{\cpt}{\mathrm{cpt}}

\newtheorem{theorem}{Theorem}[section]

\newtheorem{lemma}[theorem]{Lemma}

\newtheorem{proposition}[theorem]{Proposition}
\newtheorem{corollary}[theorem]{Corollary}
\newtheorem{conjecture}[theorem]{Conjecture}

\theoremstyle{definition}
\newtheorem{definition}[theorem]{Definition}

\newtheorem{warning}[theorem]{Warning}
\newtheorem{example}[theorem]{Example}



\numberwithin{equation}{section}

\def\Xint#1{\mathchoice
{\XXint\displaystyle\textstyle{#1}}%
{\XXint\textstyle\scriptstyle{#1}}%
{\XXint\scriptstyle\scriptscriptstyle{#1}}%
{\XXint\scriptscriptstyle\scriptscriptstyle{#1}}%
\!\int}
\def\XXint#1#2#3{{\setbox0=\hbox{$#1{#2#3}{\int}$ }
\vcenter{\hbox{$#2#3$ }}\kern-.6\wd0}}

\def\dashint{\Xint-}

\usepackage[backend=bibtex,style=alphabetic,giveninits=true]{biblatex}

\renewbibmacro{in:}{}
\DeclareFieldFormat{pages}{#1}

\begin{document}
\begin{abstract}
Motivated by Thurston and Daskalopoulos--Uhlenbeck's approach to Teichm\"uller theory, we study the behavior of $q$-harmonic functions and their $p$-harmonic conjugates in the limit as $q \to 1$, where $1/p + 1/q = 1$.
The $1$-Laplacian is already known to give rise to laminations by minimal hypersurfaces; we show that the limiting $p$-harmonic conjugates converge to calibrations $F$ of the laminations.
Moreover, we show that the laminations which are calibrated by $F$ are exactly those which arise from the $1$-Laplacian.
We also explore the limiting dual problem as a model problem for the optimal Lipschitz extension problem, which exhibits behavior rather unlike the scalar $\infty$-Laplacian.
In a companion work, we will apply the main result of this paper to associate to each class in $H^{d - 1}$ a lamination in a canonical way, and study the duality of the stable norm on $H_{d - 1}$.
\end{abstract}

\maketitle

%%%%%%%%%%%%%%%%%%%%%%%%%%%%%%%%%%%%%%%%%%%%%%%%%%%%%%%
\section{Introduction}
\subsection{\texorpdfstring{$L^\infty$}{L-infinity} variational systems}
A basic problem in metric geometry is the \dfn{optimal Lipschitz extension problem}: given two metric spaces $M, N$ and a class $\mathscr F$ of Lipschitz maps $M \to N$, find $u \in \mathscr F$ which ``stretches $M$ as little as possible.''
When $M$ is a compact euclidean domain, $N = \RR$, and $\mathscr F$ is the set of Lipschitz extensions of a Lipschitz function on $\partial M$, then this problem is solved by the $\infty$-Laplacian \cite{Crandall2008}.
To be more precise, given a Lipschitz function $f$ on $\partial M$, the unique viscosity solution of
\begin{subequations}
\begin{empheq}[left=\empheqlbrace]{align}
&	\langle \nabla^2 u, \dif u \otimes \dif u\rangle = 0, \\
&	u|_{\partial M} = f
\end{empheq}
\end{subequations}
is the unique Lipschitz extension of $f$ to $M$ which \dfn{absolutely minimizes} its Lipschitz constant, in the sense that for every open $U \subseteq M$ and every Lipschitz function $v: U \to \RR$ such that $u|_{\partial U} = v|_{\partial U}$,
$$\Lip(u|_U) \leq \Lip(v).$$

This approach does not work if $N$ is a complete Riemannian manifold of dimension $\geq 2$.
As intimated by the use of viscosity solutions, the study of the $\infty$-Laplacian is based on the maximum principle, which only applies when $N = \RR$.
Daskalopoulos and Uhlenbeck proposed to study ``$\infty$-harmonic maps'', which are weak limits of solutions of a certain generalized $p$-Laplacian as $p \to \infty$ \cite{daskalopoulos2022}.
One has little pointwise control of the $\infty$-harmonic map, so at this time it is not known if $\infty$-harmonic maps locally minimize their Lipschitz constants.

Sundry other definitions of $\infty$-harmonic maps, also known as ``tight maps'', were considered by Sheffield and Smart, as well as Jensen \cite{Sheffield12}.
These definitions emphasize local minimality of the Lipschitz constant, but it is not known if such maps exist, nor what their relation with the Daskalopolous--Uhlenbeck $\infty$-harmonic maps is.
One particular difficulty is that if $u$ is a tight holomorphic function, then there is no reasonable PDE that $u$ should solve.

It is natural to look at other $L^\infty$ variational systems, besides the optimal Lipschitz extension problem, and hope that they serve as an easier toy problem.
In this paper, we study the problem of minimizing the $L^\infty$ norm of a closed $d - 1$-form.
Motivated by Daskalopoulos--Uhlenbeck $\infty$-harmonic maps, we consider solutions of the generalized $p$-Laplacian
\begin{subequations}\label{p tight intro}
\begin{empheq}[left=\empheqlbrace]{align}
&\dif F = 0, \\ 
&\dif^*(|F|^{p - 2} F) = 0.
\end{empheq}
\end{subequations}
We call such forms \dfn{$p$-tight}; they have recently found use in the computation of $L^p$ cohomology \cite{stern2024lp}.

We are interested in the weak limits of $p$-tight forms as $p \to \infty$.
We call such a form \dfn{tight}.

Tight forms are easier to understand than $\infty$-harmonic maps for two reasons.
For one, can study nonsmooth tight forms by passing to the convex dual problem, which is the \dfn{least gradient problem}: the problem of minimizing the total variation $\int_M \star |\dif v|$ of a function $v$ \cite{gorny2024leastgradient}.
The convex dual problem for $\infty$-harmonic maps into a symmetric space is a Lie algebra-valued $1$-Laplacian system \cite{daskalopoulos2022}, which presumably enjoys similar properties to the least gradient problem but has not been studied as much.
We show that on a closed manifold, every tight form is conjugate to a least gradient function in the following sense:

\begin{theorem}\label{existence of infinity tight forms}
Let $M$ be a closed oriented Riemannian manifold of dimension $d \geq 2$, and let $\rho \in H^{d - 1}(M, \RR)$ be a nonzero cohomology class.
Then there is a tight form $F$ which represents $\rho$, and minimizes its $L^\infty$ norm among all representatives of $\rho$.
Moreover, there exists a nonconstant function of least gradient $u$ on the universal cover $\tilde M$, such that $\dif u$ descends to $M$, the product of distributions $\dif u \wedge F$ is well-defined, and
\begin{equation}\label{max flow min cut}
\dif u \wedge F = \|F\|_{L^\infty} \star |\dif u|.
\end{equation}
\end{theorem}

Since least gradient functions have a geometric interpretation as area-minimizing measured laminations \cite{BackusCML}, tight forms have an appealing interpretation as particularly well-behaved measurable calibrations which certificate that the dual lamination is area-minimizing.
In particular, up to a rescaling, every tight form calibrates the leaves of an area-minimizing lamination.

% A localized version of Theorem \ref{existence of infinity tight forms}, combined with a calibration argument, allows us to show that tight forms are absolute minimizers:

% \begin{theorem}\label{absolute minimizer 1}
% Let $M$ be an oriented Riemannian manifold of dimension $d \geq 2$ (possibly with boundary), and let $F \in L^\infty(M, \Omega^{d - 1}_{\rm cl})$ be a tight form.
% Then for every precompact open set $U$ with Lipschitz boundary and $H^1(U, \RR) = 0$,
% $$\|F\|_{L^\infty(U)} = \|F\|^{L^\infty(\partial U)}.$$
% \end{theorem}

% Here, $\|F\|_{L^\infty(\partial U)}$ is the supremum along $\partial U$ of the upper-semicontinuous envelope of $|F|$.
% The cohomological assumption in Theorem \ref{absolute minimizer 1} is used to recover a dual function of least gradient from its exterior derivative.
% This assumption cannot be removed, as we show by an example.

Unlike the $\infty$-harmonic maps, which are not associated to a well-defined PDE, tight forms are associated to the PDE
\begin{subequations}\label{tight Einstein}
\begin{empheq}[left=\empheqlbrace]{align}
	&\dif F = 0 \label{tight Einstein closed}, \\
	&(\star F) \wedge \dif \star F = 0, \label{tight Einstein Katzourakis}\\
	&(\nabla_\alpha F_{\beta_1 \cdots \beta_{d - 1}}) F^{\beta_1 \cdots \beta_{d - 1}} {F^\alpha}_{\gamma_1 \cdots \gamma_{d - 2}} = 0 \label{tight Einstein big},
\end{empheq}
\end{subequations}
in the sense that if the $p$-tight forms converge to a tight form $F$ in a strong enough topology, then $F$ satisfies (\ref{tight Einstein}).
We highlight (\ref{tight Einstein Katzourakis}), as it is does not have an analogue for the scalar $\infty$-Laplacian, but instead is a genuinely vector-valued phenomenon \cite{Katzourakis12}.
Together, the three equations (\ref{tight Einstein}) say that if $|F|$ attains a local maximum at $x$, then there is an integral hypersurface $N$ of $\ker(\star F)$ through $x$, and $F$ calibrates $N$.

Since $\infty$-harmonic functions are characterized as absolutely minimizing their Lipschitz constants, we shall be interested in if tight forms absolutely minimize their $L^\infty$ norms -- namely, whether a tight form minimizes its $L^\infty$ norm among all forms with the same trace in every open set.
Using the perspective that tight forms are calibrations, we prove:

\begin{theorem}\label{tight are absolute minimizers}
Let $F$ be a closed $C^1$ $d - 1$-form on a compact Riemannian manifold (possibly with boundary).
If $F$ solves (\ref{tight Einstein}), then for every open $U \subseteq M$ with $H_{d - 1}(U, \RR) = 0$,
$$\|F\|_{C^0(U)} = \|F\|_{C^0(\partial U)}.$$
In particular, $F|_U$ minimizes its $C^0$ norm among all closed forms of the same trace.
\end{theorem}

This theorem is particularly remarkable because the analogous result fails for more general $L^\infty$ variational systems \cite{Katzourakis15}; the difference is that we can think of solutions of (\ref{tight Einstein}) as calibrations.
Motivated by the fact that $\infty$-harmonic functions absolutely minimize their Lipschitz constant, we expect the conclusion of this theorem to hold for every tight form, not just those which meet strong regularity hypotheses.
We were not able to prove this, since we do not have access to the maximum principle; therefore we conjecture:

\begin{conjecture}
Let $F$ be a tight form and let $U$ be an open set such that $H_{d - 1}(U, \RR) = 0$.
Then the restriction of $|F|$ to $\overline U$ attains its maximum on $\partial U$.
\end{conjecture}

The converse to Theorem \ref{tight are absolute minimizers} does not hold; instead, absolute minimizers only satisfy (\ref{tight Einstein closed}) and (\ref{tight Einstein big}).

By the duality Theorem \ref{existence of infinity tight forms}, tight forms on a closed manifold always calibrate a hypersurface, even if they are not smooth.
In turn, we have seen that the existence of calibrated hypersurfaces is crucial to the proof of Theorem \ref{tight are absolute minimizers}.
This suggests a bold but vague conjecture: the correct definition of solution of an $L^\infty$ variational problem should be based on convex duality with a $BV$ variational problem.
An optimist for our bold conjecture can point to the recent work \cite{Mazon14,górny2021applications,górny2022dualitybased} which studies $BV$ problems using their duality with $L^\infty$ problems, as well as recent work \cite{daskalopoulos2023,katzourakis2024minimisers} which studies $L^\infty$ problems using their dual $BV$ problems.
However, the skeptic will point to Example \ref{boundaries bad}, which gives an obstruction to the existence of a conjugate least gradient function of a tight form when $\partial M$ is nonempty.

%%%%%%%%%%%%%%%%%
\subsection{Duality between calibrations and transverse measures}
The optimal Lipschitz extension problem naturally arises in Teichm\"uller theory.
Let $M$ be a closed surface of genus $g \geq 2$.
Thurston introduced a Finsler metric on the Teichm\"uller space $\mathscr T_g$ of hyperbolic structures on $M$ \cite{Thurston98}.
For any two hyperbolic structures $\rho, \sigma \in \mathscr T_g$, let $L(\rho, \sigma)$ be the infimum of Lipschitz constants of maps $f: (M, \rho) \to (M, \sigma)$ homotopic to $\id_M$.
Then the \dfn{Thurston metric} on $\mathscr T_g$ is $\log L$.

Thurston showed that his metric is intimately related to the structure of the geodesic laminations on $M$.
Any geodesic lamination $\lambda$ on $(M, \rho)$ can be deformed to a geodesic lamination on $(M, \sigma)$; the deformation will stretch each leaf of $\lambda$ by some factor which we call the \dfn{stretch ratio} of $\lambda$.
If $K(\rho, \sigma)$ is the supremum of the stretch ratios over all geodesic laminations, then Thurston showed that
\begin{equation}\label{Thurston MFMC formula}
L(\rho, \sigma) = K(\rho, \sigma).
\end{equation}
Furthermore, he constructed a \dfn{canonical maximally stretched lamination}, which every lamination which realizes the supremum in the definition of $K(\rho, \sigma)$ embeds in.

Thurston's proof of (\ref{Thurston MFMC formula}) is a \emph{tour de force} of geometric topology, but he conjectured that there should be an easier and more general proof \cite[Abstract]{Thurston98}:

\begin{quote}
I currently think that a characterization of minimal stretch maps should be possible in a considerably more general context (in particular, to include some version for all Riemannian surfaces), and it should be feasible with a simpler proof based on more general principles -- in particular, the max flow min cut principle, convexity, and $L^0 \leftrightarrow L^\infty$ duality. \dots
I also expect that this theory fits into a context including $L^p$ comparisons \dots. 
\end{quote}

This conjecture has recently been proven by Daskalopoulos and Uhlenbeck \cite{daskalopoulos2022,uhlenbeck2023noether,daskalopoulos2023}.
For our purposes, the salient parts of the proof can already been seen in the toy problem of optimal Lipschitz maps to $\Sph^1$ \cite{daskalopoulos2020transverse}.
Let $M$ be a closed hyperbolic surface of fundamental group $\Gamma$, let $\rho \in \Hom(\Gamma, \RR)$, and let $u$ be a $\rho$-equivariant $\infty$-harmonic function on the universal cover $\Hyp^2$.
Then $u$ minimizes $\Lip(u) = \|\dif u\|_{L^\infty}$ among all $\rho$-equivariant functions, and we can find a conjugate function of least gradient $v$.
The analogue of (\ref{Thurston MFMC formula}) follows from their main theorem: there is a geodesic lamination $\lambda$ of $M$, such that $|\dif u|$ attains its maximum on $\lambda$ and $|\dif v|$ is a transverse measure to $\lambda$.

If $u$ is an $\infty$-harmonic function on $\Hyp^2$ such that $\Lip(u) = 1$, and $\lambda$ is the associated geodesic lamination, then one can think of $\dif u$ as a calibration which calibrates the leaves of $\lambda$.
This theorem fits into a circle of general results suggesting a general duality theorem, where on one side we have calibrations arising from an $L^\infty$ variational problem, and on the other side we have transverse measures arising from functions of least gradient:
\begin{enumerate}
\item The author's previous work \cite{BackusCML} establishes that every function $u$ of least gradient induces a transverse measure $|\dif u|$ for the lamination $\lambda$ of area-minimizing hypersurfaces whose leaves are the level sets of $u$.
(Briefly, $\dif u$ is the ``Ruelle-Sullivan current'' defining $\lambda$.)
\item Using approximation by the $q$-Laplacian, Maz\'on, Rossi, and Segura de Le\'on \cite{Mazon14} showed that every function of least gradient on a compact domain is dual to a calibration $d - 1$-form, though they did not use this language.
\item Bangert and Cui \cite{bangert_cui_2017} showed that if $F$ is a continuous calibration $d - 1$-form on a closed manifold of dimension $d \leq 7$, and $F$ minimizes its $L^\infty$ norm in its cohomology class, then there is a measured lamination $\lambda$ on $M$ such that $F$ calibrates the leaves of $\lambda$.
\end{enumerate}

The present paper unifies the results of \cite{daskalopoulos2020transverse,bangert_cui_2017}, as well as the analogue of \cite{Mazon14} on closed manifolds.
By Theorem \ref{existence of infinity tight forms}, we can find a tight form $F$ in every nonzero cohomology class $\rho \in H^{d - 1}(M, \RR)$, and a least gradient function $u$ dual to $F$ in the sense that 
$$\dif u \wedge F = \|F\|_{L^\infty} \star |\dif u|.$$
If $\kappa$ is a measured oriented lamination, we say that $\kappa$ is \dfn{dual to $\rho$} if for some tight form $F$ representing $\rho$, $F/\|F\|_{L^\infty}$ is a calibration for the leaves of $\kappa$.
In that case, every leaf of $\kappa$ is a minimal hypersurface, and $\kappa$ minimizes its mass in its homology class.
Combining Theorem \ref{existence of infinity tight forms} with \cite[Theorem B]{BackusCML}, we conclude:

\begin{corollary}\label{existence of dual laminations}
Let $M$ be a closed oriented Riemannian manifold of dimension $2 \leq d \leq 7$.
For every nonzero cohomology class $\rho \in H^{d - 1}(M, \RR)$, there exists a measured oriented lamination which is dual to $\rho$.
\end{corollary}

We then give a characterization of dual laminations which is analogous to the equivalence (\ref{Thurston MFMC formula}) of Thurston's functionals $L$ and $K$.

\begin{theorem}\label{lams are calibrated}
Let $M$ be a closed oriented Riemannian manifold of dimension $2 \leq d \leq 7$.
Let $\rho \in H^{d - 1}(M, \RR)$ be a nonzero cohomology class, and let $\kappa$ be a measured oriented lamination.
Let $\Comass(\rho)$ be the minimal $L^\infty$ norm among representatives of $\rho$.
The following are equivalent:
\begin{enumerate}
\item $\kappa$ is dual to $\rho$.
\item With the supremum ranging over measured oriented laminations $\lambda$,
\begin{equation}\label{duality between stable and comass}
\Comass(\rho) = \sup_\lambda \frac{\langle \rho, [\lambda]\rangle}{\Mass(\lambda)} = \frac{\langle \rho, [\kappa]\rangle}{\Mass(\kappa)}.
\end{equation}
\item There exists a function of least gradient $u$ on $\tilde M$, such that $\dif u$ descends to $M$, $\dif u$ is the Ruelle-Sullivan current defining $\kappa$, and the Poincar\'e dual $\sigma$ of $[\dif u]$ satisfies
$$\langle \rho, \sigma\rangle = \Comass(\rho) \Mass(\dif u).$$
\end{enumerate}
\end{theorem}

In (\ref{duality between stable and comass}), the pairing is between cohomology and homology, and $\Mass(\lambda)$ is the total mass of the transverse measure to $\lambda$.
The duality relation (\ref{duality between stable and comass}) has already been interpreted in the string theory literature as a continuous analogue of the max flow min cut theorem \cite{Freedman_2016}, at least in the case that $\lambda$ has only one leaf.
On the other hand, one can think of $\langle \rho, [\lambda]\rangle/\Mass(\lambda)$ as the amount that $\lambda$ is ``stretched'' by a representative of $\rho$, or in other words we can think of it as analogous to Thurston's functional $K$ \cite[\S5.3]{daskalopoulos2020transverse}.
So (\ref{duality between stable and comass}) is an analogue of Thurston's formula (\ref{Thurston MFMC formula}), and Theorem \ref{lams are calibrated} can be viewed as an explanation for why the max flow min cut theorem should appear in Thurston's conjecture.

Theorem \ref{lams are calibrated} immediately implies the theorem of Bangert and Cui \cite{bangert_cui_2017} and the theorem of Daskalopoulos and Uhlenbeck on $\infty$-harmonic functions \cite{daskalopoulos2020transverse}.
In addition to clarifying the role of the $p$-Laplacian, Theorem \ref{lams are calibrated} is stronger than the Bangert--Cui theorem because we need not assume that $\rho$ contains a continuous form which minimizes its $L^\infty$ norm.
In general, it seems unlikely that a cohomology class necessarily contains such a form, though this is known to hold when $d = 2$ \cite{Evans08} or when $M$ is Ricci-flat. 

The following result is analogous to the main theorem of \cite{Mazon14}, and follows easily from Theorem \ref{lams are calibrated}:

\begin{corollary}\label{every minimizer is dual}
Let $M$ be a closed oriented Riemannian manifold of dimension $2 \leq d \leq 7$, and let $\kappa$ be a measured oriented lamination in $M$.
If $\kappa$ minimizes its mass in its homology class, then there exists a nonzero cohomology class $\rho \in H^{d - 1}(M, \RR)$ such that $\kappa$ is dual to $\rho$.
\end{corollary}

A companion paper \cite{BackusBest2} will make the analogy between Theorem \ref{lams are calibrated} and Thurston's approach to Teichm\"uller theory even stronger.
Thurston used his canonical maximally stretched lamination to study the duality of the tangent and cotangent bundles of Teichm\"uller space \cite[\S10]{Thurston98}.
In our situation, while there may be many measured stretch laminations associated to $\rho$, we will embed them all in a \dfn{canonical calibrated lamination} $\lambda(\rho)$.
Moreover, we will see that $\lambda(\rho)$ can be used to prove facts about the duality between the costable and stable norms, some of which were claimed but not proven by Auer and Bangert \cite{Auer01}.

The function of least gradient in Theorem \ref{existence of infinity tight forms} was extracted as a limit of $q$-harmonic conjugates as $q \to 1$.
In contrast, the function of least gradient in Theorem \ref{lams are calibrated} need not arise as the limit of $q$-harmonic conjugates, as we show by an example in \S\ref{dual lamination sec}.
On closed hyperbolic surfaces, the homology class $[\lambda]$ (which is the Poincar\'e dual of the cohomology class of the derivative of the function of least gradient) need not be unique, by a theorem of Massart \cite{Massart1997StableNO}.
It would be interesting to know if the cohomology class of the least gradient function arising from the $q$-harmonic functions in Theorem \ref{existence of infinity tight forms} is somehow privileged among all cohomology classes which are dual to a given class:

\begin{conjecture}
Let $M$ be a closed oriented Riemannian manifold of dimension $2 \leq d \leq 7$.
Let $\rho \in H^{d - 1}(M, \RR)$ be a nonzero cohomology class, and for each $1 < p < \infty$, let $F_p$ be the $p$-tight form representing $\rho$.
Let $\alpha_q \in H^1(M, \RR)$ be the cohomology class of $|F_p|^{p - 2} \star F_p$.
Then $(\alpha_q)$ has a unique limit as $q \to 1$.
\end{conjecture}

Let us conclude with an interesting problem which this work does not address.
The proof of Theorem \ref{lams are calibrated} is based on the fact that the level sets of a function of least gradient are area-minimizing hypersurfaces, but one can seek to formulate such a result for eigenfunctions of the $1$-Laplacian instead.
The first eigenfunction of the $1$-Laplacian is related to \dfn{Cheeger sets} -- sets $U$ which maximize the isoperimetric ratio $\vol_{d - 1}(\partial U)/\vol_d(U)$ \cite{Kawohl2003}.
Greiser proved a max flow min cut theorem for Cheeger sets, based on a dual $d - 1$-form \cite{Grieser05}, and it would be interesting to see if the dual $d - 1$-form solved an $L^\infty$ variational problem.

%%%%%%%%%%%%%%%%%%%%%%
\subsection{Outline of the paper}
In \S\ref{prelims}, we discuss preliminaries on calibrated geometry and functions of least gradient.

In \S\ref{tight forms sec}, we introduce $p$-tight forms and tight forms.
We then prove Theorem \ref{existence of infinity tight forms}, the existence of tight forms and their dual least gradient functions.
Finally, we show that the analogue of Theorem \ref{existence of infinity tight forms} for the Dirichlet problem fails.

In \S\ref{comass sec}, we discuss measurable calibrations of laminations.
Then we prove Theorem \ref{lams are calibrated}, the duality theorem for transverse measures and calibrations.
As a corollary, we deduce that calibrated laminations are Lipschitz.

% In \S\ref{AbsMin}, we show that tight forms minimize their comass on every domain meeting a cohomological criterion.
% The key new idea is a localization of the duality and renormalization developed in \S\ref{tight forms sec}.
% We use a counterexample to show that this hypothesis cannot be removed.

In \S\ref{infinityMax}, we study the Euler-Lagrange equation for a smooth tight form.
We formally derive this equation, and then prove Theorem \ref{tight are absolute minimizers}, the optimality condition for solutions of the Euler-Lagrange equation.
Then we construct solutions of this equation which are not tight.

In \S\ref{duality derivation}, we explain how to derive the PDE for a $p$-tight form from the Fenchel--Rockafellar duality theorem of convex optimization.
This derivation motivates the proof of Theorem \ref{existence of infinity tight forms}, but since it is not strictly necessary to understand its proof, we include it separately.

We also include Appendix \ref{GMT appendix}, in which we formulate the relevant parts of Anzellotti's theory of compensated compactness in a diffeomorphism-invariant way, and prove miscellaneous lemmata about Radon measures we shall need.

%%%%%%%%%%%%%%%%%%%%%%
\subsection{Acknowledgements}
I would like to thank Georgios Daskalopoulos and Karen Uhlenbeck for helpful discussions and for providing me with a draft copy of \cite{daskalopoulos2023}.
I would also like to thank Nikos Katzourakis for bringing my attention to \cite{Katzourakis12}.
In addition, this work has benefited from discussions by Brian Freidin, Anatole Gaudin, Bernd Kawohl, Taylor Klotz, and Trent Lucas, and I wouldl ike to thank them all for helping me make sense of the unusual mix of analysis and geometry studied here.

This research was supported by the National Science Foundation's Graduate Research Fellowship Program under Grant No. DGE-2040433.

%%%%%%%%%%%%%%%%%%%%%%%%%%%%%
\section{Preliminaries}\label{prelims}
\subsection{Notation}
Let $M$ be a Riemannian manifold of dimension $d$.
The operator $\star$ is the Hodge star on $M$, thus $\star 1$ is the Riemannian measure of $M$.
We denote the musical isomorphisms by $\sharp, \flat$.
To avoid confusion, we write $H^\ell$ for de Rham cohomology, but never a Sobolev space, which we instead denote $W^{\ell, p}$.

% Let $\mathscr F$ be a subpresheaf of the sheaf of distributional sections of some Riemannian vector bundle $E \to M$ and let $\mathcal X$ be a function space.
% We write $\mathcal X(\cdot, \mathscr F)$ for the presheaf of sections $u$ of $\mathscr F$ such that for every smooth unit-length local section $e$ of $E$, $\langle e, u\rangle \in \mathcal X$.
% We write $\mathcal X_\cpt$ for the space of compactly supported functions in $\mathcal X$, and $\mathcal X_\loc$ for the space of functions $u$ such that for every $\chi \in C^\infty_\cpt$, $\chi u \in \mathcal X$.

The sheaf of $\ell$-forms is denoted $\Omega^\ell$, and the sheaf of closed $\ell$-forms is denoted $\Omega^\ell_{\rm cl}$.
We assume that $\ell$-forms are $L^1_\loc$, but \emph{not} that they are continuous; hence $\dif$ must be meant in the sense of distributions.

We write $A \lesssim_\theta B$ to mean that $A \leq CB$, where $C > 0$ is a constant that only depends on $\theta$.
We write $A \ll_\theta B$ to mean that, as $B \to 0$, $A \to 0$, where the rate of convergence only depends on $\theta$.

%%%%%%%%%%%%%%%
\subsection{Integration of currents}\label{homological integration}
By an $\ell$-\dfn{current of locally finite mass} $T$ we mean a continuous linear functional on the space $C^0_\cpt(M, \Omega^{d - \ell})$ of continuous $d - \ell$-forms of compact support.
By a \dfn{current} we shall always mean a current of locally finite mass, unless explicitly stated otherwise.
We write $\int_M T \wedge \varphi$ for the dual pairing of a current $T$ and a form $\varphi$.
If $T$ is a current (of locally finite mass), then the components of $T$ are Radon measures, and conversely a Radon measure is a $d$-current.
If $\Sigma$ is a submanifold of codimension $\ell$, then we have an $\ell$-current $T_\Sigma$ defined by $\int_M T_\Sigma \wedge \varphi := \int_\Sigma \varphi$.

\begin{warning}
In our conventions, $\ell$-currents are generalizations of $\ell$-forms and submanifolds of codimension $\ell$, and they naturally define cohomology classes.
This convention is standard in algebraic geometry, but it is the opposite of the convention in geometric measure theory (where $\ell$-currents generalize submanifolds of dimension $\ell$ and naturally define homology classes).
\end{warning}

If $M$ is closed oriented, we have the Poincar\'e duality map
$$\PD: H^\ell(M, \RR) \to H_{d - \ell}(M, \RR).$$
Let $T$ be a closed $\ell$-current.
Since we can approximate $T$ by closed $\ell$-forms, $T$ defines a cohomology class $[T] \in H^\ell(M, \RR)$.
However, $T$ also is a functional on $d - \ell$-forms, so $T$ defines a homology class, which is nothing more than $\PD([T])$.

%%%%%%%%%%%%%%%%%%%%
\subsection{Calibrated geometry}
We recall calibrated geometry, which was developed by Harvey and Lawson \cite{Harvey82}.
The \dfn{comass} of a differential $k$-form $F$ is defined by 
$$\Comass(F) := \sup_{\Sigma \subset M} \frac{1}{\vol(\Sigma)} \int_\Sigma F,$$
where the supremum ranges over all oriented $k$-dimensional submanifolds $\Sigma$.
The \dfn{mass} of a $d - k$-current $T$ is
$$\Mass(T) := \sup_{\Comass(F) \leq 1} \int_M T \wedge F.$$

It is clear that $\Comass(F) \leq \|F\|_{L^\infty}$, but if $F$ is a closed $d - 1$-form, then the converse holds as well.
Indeed, we can view $F$ as the vector field $X := \star F^\sharp$, and then 
$$\Comass(F) = \sup_{\Sigma \subset M} \frac{1}{\vol(\Sigma)} \int_\Sigma X \cdot \normal_\Sigma \dif \mathcal H^{d - 1}.$$
By taking $\Sigma$ to be a small disk on which $|F|$ nearly attains its maximum, and such that $\normal_\Sigma$ nearly points in the same direction as $X$, we see that $\Comass(F) \geq \|F\|_{L^\infty} - \varepsilon$ for any $\varepsilon$.

A function $u$ has \dfn{bounded variation} if $\dif u$ is a $1$-current of finite mass.
$BV$ is the space of functions of bounded variation.
Its \dfn{total variation} is the mass $\Mass(\dif u)$, which we also write $\int_M \star |\dif u|$.

A \dfn{calibration} is a $k$-form $F$ such that $\dif F = 0$ and $\Comass(F) = 1$.
If $\Sigma$ is a $k$-dimensional submanifold, and $F$ pulls back to the Riemannian volume form of $\Sigma$, we say that $\Sigma$ is \dfn{$F$-calibrated}.
If $\Sigma$ is $F$-calibrated, then for any $k - 1$-dimensional submanifold $\Lambda$,
$$\vol(\Sigma) = \int_\Sigma F = \int_{\Sigma + \partial \Lambda} F \leq \vol(\Sigma + \partial \Lambda),$$
so that $\Sigma$ is area-minimizing.
On the other hand, if $A$ is a $k - 1$-form and $\Sigma$ is a closed $F$-calibrated submanifold, then
$$\Comass(F) = 1 = \frac{1}{\vol(\Sigma)} \int_\Sigma F = \frac{1}{\vol(\Sigma)} \int_\Sigma F + \dif A \leq \Comass(F + \dif A),$$
so $F$ minimizes its comass in its cohomology class if it calibrates a closed hypersurface.

The comass and mass induce norms on cohomology and homology.
The \dfn{stable norm} $\Mass$ on $H_k(M, \RR)$ is defined by 
$$\Mass(\theta) := \inf_{\PD([T]) = \theta} \Mass(T),$$
where $T$ ranges over $d - k$-currents.
Its dual norm is the \dfn{costable norm} $\Comass$ on $H^k(M, \RR)$, which is the infimum over $k$-forms of their comasses.
These norms were introduced by Federer \cite{Federer1974} and studied further by Gromov \cite{gromov2007metric}.
Work of Auer and Bangert \cite{Auer01}, which was later used by \cite{bangert_cui_2017}, shows that the stable norm has a particularly rich duality theory in codimension $1$.

%%%%%%%%%%%%%%%%%%%%
\subsection{Equivariant least gradient functions}\label{equivariance}
Let $M$ be a closed oriented manifold of dimension $d \geq 2$ and fundamental group $\Gamma$, and let $\alpha \in \Hom(\Gamma, \RR)$.
Since $\RR$ is abelian, we have the factorization
% https://q.uiver.app/#q=WzAsMyxbMCwwLCJcXEdhbW1hIl0sWzIsMCwiXFxtYXRoYmYgUiJdLFswLDIsIlxcR2FtbWEvW1xcR2FtbWEsIFxcR2FtbWFdIl0sWzAsMSwiXFxhbHBoYSJdLFswLDIsIiIsMix7InN0eWxlIjp7ImhlYWQiOnsibmFtZSI6ImVwaSJ9fX1dLFsyLDFdXQ==
\[\begin{tikzcd}
	\Gamma & {\mathbf R} \\
	{\Gamma/[\Gamma, \Gamma]}
	\arrow["\alpha", from=1-1, to=1-2]
	\arrow[two heads, from=1-1, to=2-1]
	\arrow[from=2-1, to=1-2]
\end{tikzcd}\]
and we have the Hurcewiz isomorphism $\Gamma/[\Gamma, \Gamma] = H_1(M, \RR)$.
Thus we obtain an element of 
$$H^1(M, \RR) = \Hom(H_1(M, \RR), \RR)$$
which we also call $\alpha$.
In particular, we obtain a harmonic $1$-form, which we also call $\alpha$.
In this situation, we will study functions $u$ on the universal cover $\tilde M$, which are \dfn{$\alpha$-equivariant} in the sense that for every deck transformation $\gamma \in \Gamma$, and $x \in \tilde M$,
$$u(\gamma x) = u(x) + \alpha(\gamma).$$
Then, $\dif u$ descends to a closed $1$-form on $M$, which we also call $\dif u$, and $[\dif u] = \alpha$.

Let $1 < q < \infty$, and $\alpha \in \Hom(\pi_1(M), \RR)$.
If a function $u: \tilde M \to \RR$ is $\alpha$-equivariant, then $\dif u$ descends to a closed $1$-form on $M$ whose cohomology class is (the abelianization of) $\alpha$.
We shall be interested in the $q$-Laplacian
\begin{subequations}\label{q harmonic}
\begin{empheq}[left=\empheqlbrace]{align}
&\dif^*(|\dif u|^{q - 2} \dif u) = 0, \\
&[\dif u] = \alpha
\end{empheq}
\end{subequations}
in the limit $q \to 1$.

\begin{definition}
Let $u \in BV(\tilde M, \RR)$ be a $\pi_1(M)$-equivariant function.
Suppose that, for every $\pi_1(M)$-invariant function $v$,
$$\Mass(\dif u) \leq \Mass(\dif u + \dif v).$$
Then $u$ has \dfn{least gradient}.
\end{definition}

We now give a more geometric interpretation of least gradient functions, using laminations.
See \cite{BackusCML} for a more careful formulation of laminations:

\begin{definition}
Let $I \subset \RR$ be an interval and $J \subset \RR^{d - 1}$ a box. Then:
\begin{enumerate}
\item A (codimension-$1$, Lipschitz) \dfn{laminar flow box} is a Lipschitz coordinate chart $\Psi: I \times J \to M$ and a compact set $K \subseteq I$, called the \dfn{local leaf space}, such that for each $k \in K$, $\Psi|_{\{k\} \times J}$ is a $C^1$ embedding, and the leaf $\Psi(\{k\} \times J)$ is a $C^1$ complete hypersurface in $\Psi(I \times J)$.
\item Two laminar flow boxes belong to the same \dfn{laminar atlas} if the transition map preserves the local leaf spaces.
\item A \dfn{lamination} $\lambda$ is a closed nonempty set $\supp \lambda$ and a maximal laminar atlas $\{(\Psi_\alpha, K_\alpha): \alpha \in A\}$ such that
$$\supp \lambda \cap \Psi_\alpha(I \times J) = \Psi_\alpha(K_\alpha \times J).$$
\end{enumerate}
\end{definition}

\begin{definition}
Let $\lambda$ be a lamination with laminar atlas $\{(\Psi_\alpha, K_\alpha): \alpha \in A\}$. Then:
\begin{enumerate}
\item If every leaf of $\lambda$ has zero mean curvature, then $\lambda$ is \dfn{minimal}. If in addition $d = 2$, then $\lambda$ is \dfn{geodesic}.
\item Let $F$ be a calibration. If every leaf of $\lambda$ is $F$-calibrated, then $\lambda$ is \dfn{$F$-calibrated}.
\item If the transition maps $\Psi_\alpha^{-1} \circ \Psi_\beta$ are orientation-preserving, then $\lambda$ is \dfn{oriented}.
\item A \dfn{transverse measure} $\mu$ to $\lambda$ consists of Radon measures $\mu_\alpha$ on each local leaf space $K_\alpha$, such that the transition maps $\Psi_\alpha^{-1} \circ \Psi_\beta$ send $\mu_\beta$ to $\mu_\alpha$, and $\supp \mu_\alpha = K_\alpha$.
The pair $(\lambda, \mu)$ is a \dfn{measured lamination}.
\item Suppose that $\lambda$ is measured oriented, and let $\{\chi_\alpha: \alpha \in A\}$ be a partition of unity subordinate to $\{\Psi_\alpha(I \times J): \alpha \in A\}$. The \dfn{Ruelle-Sullivan current} $T_\lambda$ acts on $\varphi \in C^0_\cpt(M, \Omega^{d - 1})$ by 
$$\int_M T_\lambda \wedge \varphi := \sum_{\alpha \in A} \int_{K_\alpha} \left[\int_{\{k\} \times J} (\Psi_\alpha^{-1})^* (\chi_\alpha \varphi)\right] \dif \mu_\alpha(k).$$
\item Suppose that $\lambda$ is measured oriented. The \dfn{mass} $\Mass(\lambda)$ is $\Mass(T_\lambda)$, and the \dfn{homology class} $[\lambda] \in H_{d - 1}(M, \RR)$ is $\PD([T_\lambda])$.
\end{enumerate}
\end{definition}

\begin{theorem}\label{1 harmonic is MOML}
Suppose that $d \leq 7$ and $u$ is a $\pi_1(M)$-equivariant function on $\tilde M$.
The following are equivalent:
\begin{enumerate}
\item $u$ has least gradient.
\item There is a measured oriented minimal lamination $\lambda_u$, which minimizes its mass in its homology class, whose leaves are all of the form $\partial \{u > y\}$ or $\partial \{u < y\}$, $y \in \RR$, and whose Ruelle-Sullivan current is $\dif u$.
\end{enumerate}
\end{theorem}
\begin{proof}
If $u$ has least gradient, then \cite[Theorem B]{BackusCML} implies that there is a measured oriented minimal lamination $\tilde \lambda_u$ on $\tilde M$ whose leaves are level sets of $u$, and whose Ruelle-Sullivan current is $\dif u$.
Since $u$ is equivariant, $\tilde \lambda_u$ descends to a lamination $\lambda_u$ on $M$ such that $\Mass(\lambda_u) = \Mass(\dif u)$.
Since $u$ has least gradient, $\lambda_u$ is mass-minimizing.

Conversely, if such a lamination exists, \cite[Theorem B]{BackusCML} implies that $u$ locally has least gradient and $\Mass(\dif u) = \Mass(\lambda_u)$, so $u$ has least gradient.
\end{proof}

%%%%%%%%%%%%%%%%%%%%%%%%%%%%%
\section{Tight forms and functions of least gradient}\label{tight forms sec}
In this section we shall prove Theorem \ref{existence of infinity tight forms}, by studying properties of $p$-tight forms.
Unless noted otherwise, $M$ shall denote a closed oriented Riemannian manifold, $\tilde M \to M$ shall be the universal covering, $M_{\rm fun} \subseteq \tilde M$ shall be a fundamental domain of $M$, and $\Gamma := \pi_1(M)$.

%%%%%%%%%%%%%%%%%%%%%%%%%%%
\subsection{Equivariant \texorpdfstring{$q$-harmonic}{q-harmonic} functions}
Let $\alpha \in \Hom(\Gamma, \RR)$ and $1 < q < \infty$.
Abelianizing $\alpha$, we think of $\alpha$ as a cohomology class in $H^1(M, \RR)$.
We are interested in $\alpha$-equivariant $q$-harmonic functions $u$ on $\tilde M$, or equivalently closed $1$-forms $\dif u$ on $M$ which minimize their $L^q$ norm in the cohomology class $\alpha$.

We first estimate the energy of an $\alpha$-equivariant $q$-harmonic function $u$ in terms of the stable norm $\Mass(\PD(\alpha))$ and the geometry of $M$.
If $M$ is a hyperbolic surface, this estimate is due to Massart \cite[\S4.2]{Massart96}.
We introduce the \dfn{maximal intersection number}
$$i_M := \sup_{\substack{\xi \in C^\infty(M, \Omega^1) \\ \eta \in C^\infty(M, \Omega^{d - 1})}} \frac{1}{\|\xi\|_{L^1} \|\eta\|_{L^1}} \int_M \xi \wedge \eta$$
which is a positive, finite constant that only depends on the Riemannian manifold $M$.

\begin{lemma}
Let $1/p + 1/q = 1$, and let $u_q: \tilde M \to \RR$ be an $\alpha$-equivariant $q$-harmonic function.
Then
\begin{equation}\label{q Laplacian Sobolev regularity estimate}
\vol(M)^{-1/p} \Mass(\PD(\alpha)) \leq \|\dif u_q\|_{L^q} \leq \vol(M)^{1/q} i_M \Mass(\PD(\alpha))
\end{equation}
\end{lemma}
\begin{proof}
Let $u_\infty$ be an $\alpha$-equivariant function which minimizes its Lipschitz constant $\|\dif u_\infty\|_{L^\infty}$ among all $\alpha$-equivariant functions. (For example, we can take $u_\infty$ to be the equivariant $\infty$-harmonic function constructed by Daskalopoulos and Uhlenbeck \cite[\S2]{daskalopoulos2020transverse}.)
Then, since $\dif u_q$ minimizes its $L^q$ norm, we obtain from H\"older's inequality
$$\|\dif u_q\|_{L^q} \leq \|\dif u_\infty\|_{L^q} \leq \vol(M)^{1/q} \|\dif u_\infty\|_{L^\infty}.$$
By the duality of $L^1$ and $L^\infty$,
$$\|\dif u_\infty\|_{L^\infty} = \sup_{\eta \in C^\infty(M, \Omega^{d - 1})} \frac{1}{\|\eta\|_{L^1}} \int_M \dif u_\infty \wedge \eta.$$
If we now let $v_\varepsilon$ be $\alpha$-equivariant with $\|\dif v_\varepsilon\|_{L^1} \leq \Mass(\PD(\alpha)) + \varepsilon$, then $u - v_\varepsilon$ descends to a function on $M$, so by Stokes' theorem,
\begin{align*}
\sup_{\eta \in C^\infty(M, \Omega^{d - 1})} \frac{1}{\|\eta\|_{L^1}} \int_M \dif u_\infty \wedge \eta 
&= \sup_{\eta \in C^\infty(M, \Omega^{d - 1})} \frac{1}{\|\eta\|_{L^1}} \int_M \dif v_\varepsilon \wedge \eta\\
&\leq i_M \|\dif v_\varepsilon\|_{L^1} \\
&\leq i_M(\Mass(\PD(\alpha)) + \varepsilon).
\end{align*}
Taking $\varepsilon \to 0$, we deduce one direction of (\ref{q Laplacian Sobolev regularity estimate}).
In the other direction, we estimate using H\"older's inequality
\begin{align*}
\Mass(\PD(\alpha)) &\leq \|\dif u_q\|_{L^1} \leq \|\dif u_q\|_{L^q} \vol(M)^{1/p}. \qedhere 
\end{align*}
\end{proof}

Since we are interested in the limit $q \to 1$, we must show that convergence of equivariant functions, even in a very weak function space, implies convergence of the representations.

\begin{lemma}\label{L1 convergence preserves pi1}
For each $1 < q \leq 2$, let $\alpha_q$ be a representation, let $u_q$ be an $\alpha_q$-equivariant function on $\tilde M$, and let $u \in L^1_\loc(\tilde M, \RR)$.
Suppose that $u_q \to u$ in $L^1_\loc$.
Then:
\begin{enumerate}
\item $\alpha_q \to \alpha$ for some representation $\alpha$.
\item $u$ is $\alpha$-equivariant.
\item If $\dif u_q \weakto^* \dif u$ and $\dif u_q \in L^q$, then
\begin{equation}\label{q to 1 Holder}
\Mass(\dif u) \leq \liminf_{q \to 1} \frac{1}{q} \int_M \star |\dif u_q|^q.
\end{equation}
\end{enumerate}
\end{lemma}
\begin{proof}
For each $\gamma \in \Gamma$, let
$$U_\gamma := M_{\rm fun} \cup \gamma_* (M_{\rm fun}).$$
We claim that $(\alpha_q)$ has a convergent subsequence.
To see this, we first recall that $M$ has finite Betti numbers, so $H^1(M, \RR)$ is locally compact.
Therefore, if no convergent subsequence exists, there exists a $\gamma \in \pi_1(M)$ and a subsequence along which $\langle \alpha_q, \gamma\rangle \to \infty$.
Moreover, since $u_q \to u$ in $L^1_\loc$, $\|u_q\|_{L^1(M_{\rm fun})} \leq 2\|u\|_{L^1(M_{\rm fun})}$ if $q - 1$ is small enough.
But then 
$$\|u_q\|_{L^1(\gamma_* M_{\rm fun})} = \|\gamma^* u_q\|_{L^1(M_{\rm fun})} \geq \langle \alpha_q, \gamma\rangle - \|u_q\|_{L^1(M_{\rm fun})} \geq \langle \alpha_q, \gamma\rangle - 2\|u\|_{L^1(M_{\rm fun})}$$
and taking $q \to 1$ we conclude that $(u_q)$ is not compact in $L^1(\gamma_* M_{\rm fun})$, contradicting the convergence in $L^1_\loc(\tilde M)$.
So $\alpha_q \to \alpha$ for some $\alpha \in H^1(M, \RR)$ along a subsequence.

For any $q > 1$,
\begin{align*}
\dashint_{M_{\rm fun}} \star |\gamma^* u - u - \langle \alpha, \gamma\rangle| 
&\leq \dashint_{M_{\rm fun}} \star (|\gamma^* u_q - u_q - \langle \alpha_q, \gamma\rangle| + |\gamma^* u_q - u_q| + |\gamma^* u - u|) \\
&\qquad + |\langle \alpha_q - \alpha, \gamma\rangle|.
\end{align*}
Taking $q \to 1$, we conclude that $\|\gamma^* u - u - \langle \alpha, \gamma\rangle\|_{L^1} = 0$, hence $u$ is $\alpha$-equivariant.
Thus $\alpha$ is uniquely defined and $\alpha_q \to \alpha$ along the entire subsequence.

Finally we prove (\ref{q to 1 Holder}).
Suppose that $\dif u_q \weakto^* \dif u$ as measures, and $\dif u_q$ in $L^q$.
Then
$$\|\dif u_q\|_{L^1} = \Mass(\dif u_q).$$
So we may use Proposition \ref{portmanteau} and H\"older's inequality to estimate (where $1/p + 1/q = 1$)
\begin{align*}
\Mass(\dif u) &= \lim_{q \to 1} \Mass(\dif u_q) \leq \lim_{q \to 1} \vol(M)^{\frac{1}{p}} \|\dif u_q\|_{L^q} = \lim_{q \to 1} \frac{1}{q} \int_M \star |\dif u_q|^q. \qedhere
\end{align*}
\end{proof}

%%%%%%%%%%%%%%%%
\subsection{Duality and \texorpdfstring{$p$-tight forms}{p-tight forms}}
Let $1 < p < \infty$ and $1/p + 1/q = 1$.

\begin{definition}
A \dfn{$p$-tight form} is a solution of the PDE 
\begin{subequations}\label{p tight}
\begin{empheq}[left=\empheqlbrace]{align}
&\dif F = 0, \\ 
&\dif^*(|F|^{p - 2} F) = 0. \label{p tight 2}
\end{empheq}
\end{subequations}
\end{definition} 

In \S\ref{duality derivation}, we motivate study of the PDE (\ref{p tight}) by deriving it by applying the Fenchel--Rockafellar duality theorem to the $q$-Laplacian.

\begin{proposition}
There is a unique $p$-tight form in each cohomology class in $H^{d - 1}(M, \RR)$.
Moreover, $p$-tight forms are minimizers of the strictly convex functional
$$J_p(F) := \frac{1}{p} \int_M \star |F|^p$$
among all forms cohomologous to them.
\end{proposition}
\begin{proof}
Strict convexity of $J_p$ on closed $L^p$ $d - 1$-forms is straightforward; since each cohomology class is an affine subspace of $L^p(M, \Omega^{d - 1})$, and hence is convex, the strict convexity on each class follows.
Since $J_p(F) \to \infty$ as $\|F\|_{L^p} \to \infty$, $J_p$ is coercive on $L^p(M, \Omega^{d - 1})$.
Therefore we have existence and uniqueness \cite[Chapter II]{Ekeland99}.
To compute the Euler-Lagrange equation for $J_p$, let $B$ be a $d-2$-form (so $F + t \dif B$ is cohomologous to $F$), so that for a minimizer $F$ of $J_p$,
$$\frac{\dif}{\dif t} J_p(F + t \dif B) = \frac{1}{p} \int_M \star \frac{\partial}{\partial t} |F + t \dif B|^p = \int_M \star |F + t \dif B|^{p - 2} \langle F + t \dif B, \dif B\rangle.$$
Setting $t = 0$, we obtain 
$$0 = \int_M \star |F|^{p - 2} \langle F, \dif B\rangle = \int_M \star \langle \dif^*(|F|^{p - 2} F), B\rangle.$$
Thus the Euler-Lagrange equation for $J_p$ is (\ref{p tight 2}).
\end{proof}

In \S\ref{duality derivation}, we show that every $q$-harmonic function defines a $p$-tight form by the mapping (\ref{dual solution appendix}).
The mapping in the below definition, which sends a $p$-tight form to a $q$-harmonic function, is the inverse of (\ref{dual solution appendix}).

\begin{definition}
Let $F$ be a $p$-tight form on $M$, let
\begin{equation}
\dif u := (-1)^d |F|^{p - 2} \star F, \label{inverse extremality}
\end{equation}
and let $u$ be the primitive of $\dif u$ on the universal cover $\tilde M$, which is normalized to have zero mean on a fundamental domain $M_{\rm fun}$.
Then $u$ is called the \dfn{$q$-harmonic conjugate} of the $p$-tight form $F$.
\end{definition}

Let $u$ be the $q$-harmonic conjugate of a $p$-tight form $F$.
We record that 
\begin{equation}\label{q energy is p energy}
\int_M \star |\dif u|^q = \int_M \star |F|^p.
\end{equation}
By Poincar\'e's inequality,
$$\|u\|_{W^{1, q}(M_{\rm fun})}^q \lesssim \int_M \star |\dif u|^q = \int_M \star |F|^p < \infty$$
since $F$ is $p$-tight; that is, we have $F \in L^p(M)$ and $u \in W^{1, q}_\loc(\tilde M, \RR)$, justifying any manipulations we shall make with these forms.

Our next lemma actually follows from the derivation of (\ref{p tight}) in \S\ref{duality derivation}.
However, it is instructive to give a proof which does not use the duality theorem as a black box.

\begin{lemma}
Let $F$ be a $p$-tight form, and let $u$ be its $q$-harmonic conjugate.
Then $u$ is $q$-harmonic, $F$ satisfies
\begin{equation}\label{dual solution}
F := -|\dif u|^{q - 2} \star \dif u.
\end{equation}
 and we have
\begin{equation}\label{strong duality}
\frac{1}{q} \int_M \star |\dif u|^q + \frac{1}{p} \int_M \star |F|^p + \int_M \dif u \wedge F = 0.
\end{equation}.
\end{lemma}
\begin{proof}
We first use
$$(p - 2)(q - 1) + (q - 2) = 0$$
to prove that
$$|\dif u|^{q - 2} \star \dif u = (-1)^d |F|^{(q - 2)(p - 1)} \star \star |F|^{p - 2} F = - |F|^{(q - 2)(p - 1) - (p - 2)} F = - F.$$
Thus we have (\ref{dual solution}), and moreover
$$\dif \star (|\dif u|^{q - 2} \dif u) = - \dif F = 0$$
so that $u$ is $q$-harmonic.
By (\ref{q energy is p energy}),
$$\frac{1}{q} \int_M \star |\dif u|^q + \frac{1}{p} \int_M \star |F|^p = \left[\frac{1}{q} + \frac{1}{p}\right] \int_M \star |F|^p = \int_M \star |F|^p.$$
But
$$\int_M \dif u \wedge F = (-1)^d \int_M |F|^{p - 2} \star F \wedge F = \int_M \star |F|^p,$$
so both sides of (\ref{strong duality}) are equal to $\int_M \star |F|^p$.
\end{proof}

%%%%%%%%%%%%%%%%%%%%%%%
\subsection{\texorpdfstring{Existence of tight forms}{Existence of tight forms}}
We now show that the $p$-tight forms in a class $\rho \in H^{d - 1}(M, \RR)$ converge to a form $F$ which minimizes its comass in $\rho$.
In other words, $\Comass(F)$ equals the costable norm $\Comass(\rho)$.
Since this is such a useful condition, we define that a closed $d - 1$-form $F$ has \dfn{best comass} if $\Comass(F) = \Comass([F])$.

\begin{lemma}
Let $1 < p < \infty$, let $F_p$ be a $p$-tight form, and let $B$ range over Lipschitz $d - 2$-forms. Then
\begin{equation}\label{infinity magnetic rules p magnetic}
	\|F_p\|_{L^p} \leq \vol(M)^{1/p} \inf_B \Comass(F + \dif B)
\end{equation}
\end{lemma}
\begin{proof}
By H\"older's inequality and the fact that $F_p$ is $p$-tight,
$$\|F_p\|_{L^p} \leq \|F + \dif B\|_{L^p} \leq \vol(M)^{1/p} \|F + \dif B\|_{L^\infty} = \vol(M)^{1/p} \Comass(F + \dif B),$$
hence the same holds for the infimum.
\end{proof}

\begin{proposition}\label{existence infinity}
Let $\rho \in H^{d - 1}(M, \RR)$.
For each $1 < p < \infty$, let $F_p$ be the $p$-tight form representing $\rho$.
Then there exists a closed $d - 1$-form $F$ such that:
\begin{enumerate}
\item $F_p \to F$ weakly in $L^r$ along a subsequence, for any $1 < r < \infty$.
\item $F$ is a best comass representative of $\rho$.
\end{enumerate}
\end{proposition}
\begin{proof}
Let $G$ be an $L^\infty$ representative of $\rho$.
By H\"older's inequality and (\ref{infinity magnetic rules p magnetic}), for every $r$,
\begin{equation}\label{uniform bounds in p by best curl}
	\|F_p\|_{L^r} \leq \vol(M)^{\frac{1}{r} - \frac{1}{p}} \|F_p\|_{L^p} \leq \vol(M)^{\frac{1}{r}} \|G\|_{L^\infty}.
\end{equation}
Thus a compactness argument gives $F_p \to F$ for some $d - 1$-form $F$, weakly in $L^r$, and 
$$\|F\|_{L^r} \leq \liminf_{p \to \infty} \|F_p\|_{L^r} \leq \vol(M)^{\frac{1}{r}} \|G\|_{L^\infty}.$$
Diagonalizing, we may assume that $F_p \to F$ weakly in $L^r$ for every such $r$, and taking $r \to \infty$, we conclude 
\begin{equation}\label{infinity magnetics have best curl}
	\|F\|_{L^\infty} \leq \|G\|_{L^\infty}.
\end{equation}
Moreover, $[F] = \lim_{p \to \infty} [F_p] = \rho$.
Since $G$ was arbitrary in (\ref{infinity magnetics have best curl}), $F$ has best comass.
\end{proof}

\begin{definition}
The $d - 1$-form $F$ of best comass in Proposition \ref{existence infinity} is called a \dfn{tight form}.
\end{definition}

It is a corollary of Proposition \ref{existence infinity} that every cohomology class is represented by a form of best comass.
This could be shown more directly using Alaoglu's theorem on the weakstar topology of $L^\infty$.
However, since $p$-tight forms are intimately related to $q$-harmonic functions, it is more convenient to work with tight forms than general forms of best comass.

\begin{lemma}\label{p tights approximate L}
Let $F_p$ be the $p$-tight representative of $\rho$. Then 
$$\lim_{p \to \infty} \|F_p\|_{L^p} = \Comass(\rho).$$
\end{lemma}
\begin{proof}
We follow \cite[Lemma 2.7]{daskalopoulos2020transverse}.
Let $F$ be a best comass representative of $\rho$, so
$$\|F\|_{L^\infty} = \Comass(F) = \Comass(\rho).$$
Since $F_p$ is $p$-tight, H\"older's inequality implies 
$$\|F_p\|_{L^p} \leq \|F\|_{L^p} \leq \vol(M)^{\frac{1}{p}} \Comass(\rho).$$
Therefore 
$$\limsup_{p \to \infty} \|F_p\|_{L^p} \leq \Comass(\rho).$$
To prove the converse, suppose that for some $\varepsilon > 0$,
$$\liminf_{p \to \infty} \|F_p\|_{L^p} \leq \Comass(\rho) - \varepsilon < \Comass(\rho).$$
Along a subsequence which attains the limit inferior, $F_p$ converges weakly in every $L^r$ to a tight form $\tilde F$ such that (by H\"older's inequality)
$$\|\tilde F\|_{L^r} \leq \liminf_{p \to \infty} \|F_p\|_{L^r} \leq \liminf_{p \to \infty} \vol(M)^{\frac{1}{r}} \|\tilde F\|_{L^\infty} \leq \vol(M)^{\frac{1}{r}} (\Comass(\rho) - \varepsilon).$$
Taking $r \to \infty$, we obtain $\Comass(\tilde F) < \Comass(\rho)$, which contradicts the definition of the costable norm $\Comass(\rho)$.
\end{proof}

%%%%%%%%%%%%%%%%%%%%
\subsection{\texorpdfstring{$1$-harmonic conjugates of tight forms}{One-harmonic conjugates of tight forms}}
We now construct the $1$-harmonic conjugates of a tight form.
In the special case that the tight form $F$ is a calibration, that is $\Comass(F) = 1$, a $1$-harmonic conjugate will be a $1$-harmonic function on the universal cover whose level sets are calibrated by $F$.

\begin{definition}
Let $F$ be a tight form of cohomology class $\rho$.
A nonconstant $\Gamma$-equivariant function of least gradient $u \in BV_\loc(\tilde M)$ is called a \dfn{$1$-harmonic conjugate} of $F$ if
\begin{equation}\label{1 extremality}
\dif u \wedge F = \Comass(\rho) \star |\dif u|.
\end{equation}
\end{definition}

By Proposition \ref{Anzellotti wedge product exists}, the wedge product $\dif u \wedge F$ exists as an Anzellotti wedge product, in particular as a Radon measure, and
$$\Mass(\dif u \wedge F) \leq \Comass(F) \Mass(\dif u) = \Comass(\rho) \Mass(\dif u).$$

\begin{lemma}\label{1 extremality implies least gradient}
Let $F$ be a tight form, and let $u$ be a $1$-harmonic conjugate of $F$.
Then $u$ has least gradient.
\end{lemma}
\begin{proof}
This is essentially a calibration argument.
Let $v$ be $\Gamma$-invariant.
Then, by (\ref{1 extremality}), Stokes' theorem, and the fact that $F$ has best comass,
$$\Mass(\dif u) \Comass(\rho) = \int_M \dif u \wedge F = \int_M (\dif u + \dif v) \wedge F \leq \Mass(\dif u + \dif v) \Comass(F) = \Mass(\dif u + \dif v) \Comass(\rho).$$
We conclude that $\Mass(\dif u) \leq \Mass(\dif u + \dif v)$.
\end{proof}

We are going to derive (\ref{1 extremality}) as a limit of the duality relation (\ref{strong duality}), which itself was a consequence of the convex duality theorem.
But (\ref{inverse extremality}) blows up as $p \to \infty$, so we must ``renormalize'' it before taking the limit $q \to 1$, as in \cite[\S3.2]{daskalopoulos2020transverse}.
This goal is accomplished by the next lemma.

\begin{lemma}\label{normalizations converge}
Let $\rho \in H^{d - 1}$, and let $k_p$ satisfy 
$$k_p^{1 - p} = \int_M \star |F_p|^p$$
where $F_p$ is the $p$-tight representative of $\rho$.
Then, as $p \to \infty$, $k_p \to \Comass(\rho)^{-1}$.
\end{lemma}
\begin{proof}
We follow \cite[Lemma 3.4]{daskalopoulos2020transverse}.
By Lemma \ref{p tights approximate L},
$$\lim_{p \to \infty} k_p^{-\frac{1}{q}} = \lim_{p \to \infty} \|F_p\|_{L^p} = \Comass(\rho).$$
Taking logarithms we see that $q^{-1} \log k_p \to -\log \Comass(\rho)$, and since $q \to 1$ the claim follows.
\end{proof}

Theorem \ref{existence of infinity tight forms} is the conjunction of the following theorem and Proposition \ref{existence infinity}.

\begin{theorem}\label{existence 1}
Let $\rho \in H^{d - 1}(M, \RR)$ be nonzero, let $F_p$ be its $p$-tight representative, and let 
$$F := \lim_{p \to \infty} F_p$$
be the tight representative of $F$.
Let $1/p + 1/q = 1$, and let $u_q$ be the function on $\tilde M$ with mean zero on $M_{\rm fun}$ and
$$\dif u_q = (-1)^{d - 1} k_p^{p - 1} |F_p|^{p - 2} \star F_p.$$
Then there exists an equivariant function $u \in BV_\loc(\tilde M)$ such that:
\begin{enumerate}
\item Along a subsequence as $q \to 1$, $u_q \weakto^* u$ in $BV_\loc(\tilde M)$, and for every $1 \leq r < \frac{d}{d - 1}$, $u_q \to u$ in $L^r(\tilde M)$.
\item $u$ is a $1$-harmonic conjugate of $F$.
\end{enumerate}
\end{theorem}
\begin{proof}
Let $L := \Comass(\rho)$.
We first compute using H\"older's inequality and Lemma \ref{normalizations converge}
\begin{align}
\lim_{q \to 1} \|\dif u_q\|_{L^1}
&\leq \lim_{q \to 1} \vol(M)^{\frac{1}{p}} \left[\int_M \star |\dif u_q|^q\right]^{\frac{1}{q}} = \lim_{p \to \infty} \left[k_p^p \int_M \star |F_p|^p\right]^{\frac{1}{q}} \label{Rellich}\\
&= \lim_{p \to \infty} k_p^{\frac{1}{q}} = \lim_{p \to \infty} k_p = \frac{1}{L} \nonumber.
\end{align}
So by Rellich's theorem, $(u_q)$ is weakly compact in $BV$ and strongly compact in $L^r$ for $1 \leq r < \frac{d}{d - 1}$.
In particular, $\dif u_q \to \dif u$ in the weak topology of measures and $u_q \to u$ weakly in $BV$ and strongly in $L^r$.
As the limit of $\Gamma$-equivariant functions, $u$ is also $\Gamma$-equivariant by Lemma \ref{L1 convergence preserves pi1}.
In particular, $\dif u$ descends to a current on $M$.
Moreover, $[\dif u_q] \to [\dif u]$, and we have the bound (\ref{q to 1 Holder}) on $\Mass(\dif u)$.

We next must check that $u$ is nonconstant.
If $u$ is constant, then it is $\Gamma$-invariant, so $[\dif u_q] \to 0$.
By (\ref{q Laplacian Sobolev regularity estimate}), $\|\dif u_q\|_{L^q} \to 0$, so by (\ref{Rellich}), $L = \infty$, which is absurd.
Therefore $u$ is nonconstant.

Renormalizing (\ref{strong duality}), we obtain 
$$\frac{k_p^{-p}}{q} \int_M \star |\dif u_q|^q + \frac{1}{p} \int_M \star |F_p|^p = k_p^{1 - p} \int_M \dif u_q \wedge F_p.$$
Multiplying by $k_p^p$, we have 
\begin{equation}\label{1 strong duality before limits}
	\frac{1}{q} \int_M \star |\dif u_q|^q + \frac{k_p^p}{p} \int_M \star |F_p|^p = k_p \int_M \dif u_q \wedge F_p.
\end{equation}

Let $\mu(U) := \Mass_U(\dif u)$ be the total variation measure of $\dif u$.
We claim that
\begin{equation}\label{1 strong duality}
	L\mu(M) \leq \int_M \dif u \wedge F.
\end{equation}
First, we have from (\ref{q to 1 Holder}) and (\ref{1 strong duality before limits}) that
$$\mu(M) \leq \lim_{q \to 1} \frac{1}{q} \int_M \star |\dif u_q|^q = \lim_{p \to \infty} k_p \int_M \dif u_q \wedge F_p - \lim_{p \to \infty} \frac{k_p^p}{p} \int_M \star |F_p|^p.$$
By Lemma \ref{normalizations converge},
$$\lim_{p \to \infty} \frac{k_p^p}{p} \int_M \star |F_p|^p = \lim_{p \to \infty} \frac{k_p}{p} = \frac{0}{L} = 0,$$
and
$$\lim_{p \to \infty} k_p \int_M \dif u_q \wedge F_p = \frac{1}{L} \lim_{p \to \infty} \int_M [\dif u_q] \wedge \rho.$$
Since $[\dif u_q] \to [\dif u]$, we obtain
$$\lim_{p \to \infty} \int_M [\dif u_q] \wedge \rho = \int_M \alpha \wedge \rho = \int_M \dif u \wedge F,$$
completing the proof of (\ref{1 strong duality}).

Localizing (\ref{Anzellotti Holder inequality}) to an open set $E \subseteq M$, we bound
$$\int_E \dif u \wedge F \leq L \mu(E).$$
Since $\mu$ is a Radon measure and $M$ is compact, every Borel set can be $\mu$-approximated from without by open sets, so the same inequality holds if $E$ is merely a Borel set.
So, by Lemma \ref{measurable function is 1} with $\star f := \dif u \wedge F/(L|\dif u|)$, (\ref{1 extremality}) holds.
In particular, by Lemma \ref{1 extremality implies least gradient}, $u$ has least gradient.
\end{proof}

%%%%%%%%%%%%%%%%
\subsection{Counterexample for the Dirichlet problem}
We now observe that the analogue of Theorem \ref{existence of infinity tight forms} -- specifically, the existence of a dual least gradient function for the tight form -- fails for the Dirichlet problem.
This is particularly remarkable because the proof of the theorem of Maz\'on, Rossi, and Segura de Le\'on \cite{Mazon14}, which characterizes least gradient functions in terms of the dual $L^\infty$ problem, actually shows that every least gradient function (for the Dirichlet problem) is the dual of some tight form.

\begin{example}\label{boundaries bad}
Let
$$v(x + iy) := \arctan\left(\frac{y}{x}\right)$$
defined on the open disk $M$ bounded by the circle $(x - 2)^2 + y^2 = 1$.

We check that $v$ is $\infty$-harmonic, so $F := \dif v$ is a tight $1$-form.
To do this, it is best to work in polar coordinates, $x + iy = re^{i\theta}$.
Then $v(re^{i\theta}) = \theta$, so $\dif v = \dif \theta$.
The euclidean metric is 
$$g = \dif r^2 + r^2 \dif \theta^2,$$
so the Christoffel symbol ${\Gamma^\theta}_{\theta \theta}$ vanishes.
Then we compute 
$$\Delta_\infty v = \langle \nabla \dif \theta, \dif \theta \otimes \dif \theta\rangle = \langle \nabla \dif \theta, \partial_\theta \otimes \partial_\theta \rangle r^{-4} = r^{-4} {\Gamma^\theta}_{\theta \theta} = 0.$$
Similarly, we compute $|\dif \theta| = r^{-1}$, which only attains its maximum at the boundary point $(x, y) = (1, 0)$.
In particular, $\|\dif v\|_{C^0} = 1$.
So if $u$ was a conjugate $1$-harmonic to $F$, in the sense that $\dif u \wedge F = |\dif u|$, then $\dif u$ would be identically zero away from $(1, 0)$, which is absurd.

A more geometric way to visualize this phenomenon is to notice that the streamlines of $v$ -- that is, the integral curves of $\nabla v$ -- are the circles centered on $0$.
If $u$ was a conjugate $1$-harmonic, then the level sets of $u$ would correspond to the streamlines of $v$.
However, since $u$ has least gradient, its level sets are straight lines.
\end{example}

\section{Calibrations of laminations}\label{comass sec}
\subsection{Measurable calibrations}\label{L infinity calibrations}
Let $M$ be an oriented Riemannian manifold of dimension $d \geq 2$, and let $F \in L^\infty(M, \Omega^{d - 1}_{\rm cl})$ be a closed form.
By Proposition \ref{integration is welldefined}, for every hypersurface $N \subset M$, $\int_N F$ is well-defined.
Similarly, if $T$ is a closed $1$-current of locally finite mass, then we can locally write $T = \dif u$ for some function $u \in BV$, and so the Anzellotti wedge product $T \wedge F$, constructed in Proposition \ref{Anzellotti wedge product exists}, is well-defined and has locally finite mass.
Thus, the following definition makes sense:

\begin{definition}
Let $F$ be a measurable calibration $d - 1$-form. Then:
\begin{enumerate}
\item A hypersurface $N \subset M$ is \dfn{$F$-calibrated} if 
$$\int_N F = \vol(N).$$
\item A lamination $\lambda$ is \dfn{$F$-calibrated} if every leaf of $\lambda$ is $F$-calibrated.
\item A closed $1$-current $T$ of locally finite mass is \dfn{$F$-calibrated} if
$$\int_M T \wedge F = \Mass(T).$$
\end{enumerate}
\end{definition}

Let $\lambda$ be a measured oriented Lipschitz lamination.
Then $\lambda$ can be identified with its Ruelle-Sullivan current $T_\lambda$, which has locally finite mass.
By $\Mass(\lambda)$ we mean $\Mass(T_\lambda)$.
Since $T_\lambda$ is closed, it has a cohomology class $[T_\lambda] \in H^1(M, \RR)$.
Therefore $\lambda$ has a homology class
$$[\lambda] := \PD([T_\lambda]) \in H_{d - 1}(M, \RR).$$

Let $(\chi_\alpha)$ be a locally finite partition of unity subordinate to a laminar atlas $(U_\alpha)$ for the measured oriented lamination $\lambda$.
Let $K_\alpha$ denote the local leaf space of $U_\alpha$, and $\mu_\alpha$ the transverse measure on $K_\alpha$.
If $\sigma_{\alpha, k}$ denotes the leaf in $U_\alpha$ corresponding to the real number $k \in K_\alpha$, then the definition of the Ruelle-Sullivan current unpacks as
\begin{equation}\label{coarea formula on laminations}
\int_M T_\lambda \wedge F = \sum_\alpha \int_{K_\alpha} \int_{\sigma_{\alpha, k}} \chi_\alpha F \dif \mu_\alpha(k).
\end{equation}
Since $T_\lambda$ and $F$ are closed, if $M$ is closed, then the left-hand side of (\ref{coarea formula on laminations}) is a homological invariant:
\begin{equation}\label{Ruelle Sullivan homology}
\int_M T_\lambda \wedge F = \langle [F], [\lambda]\rangle.
\end{equation}
We now show that $T_\lambda$ is $F$-calibrated iff $\lambda$ is:

\begin{proposition}\label{calibration condition}
Let $F$ be a calibration.
Let $T_\lambda$ be the Ruelle-Sullivan current of a measured oriented lamination $\lambda$.
Then the following are equivalent:
\begin{enumerate}
\item $T_\lambda$ is $F$-calibrated.
\item $\lambda$ is $F$-calibrated.
\end{enumerate}
\end{proposition}
\begin{proof}
First suppose that $T_\lambda$ is $F$-calibrated.
Let $(\chi_\alpha)$ be a locally finite partition of unity subordinate to an open cover $(U_\alpha)$ of flow boxes for $\lambda$, let $(K_\alpha)$ be the local leaf spaces, and let $(\mu_\alpha)$ be the transverse measure.
After refining $(U_\alpha)$ we may assume that $U_\alpha$ is a ball which satisfies the hypotheses of the regularized Poincar\'e lemma (Proposition \ref{Hodge theorem}). After shrinking $U_\alpha$ we may assume that $\chi_\alpha > 0$ on $U_\alpha$.
Then for leaves $\sigma_{\alpha,k}$, we rewrite (\ref{coarea formula on laminations}) as 
$$\Mass(\lambda) = \int_M T_\lambda \wedge F = \sum_\alpha \int_{K_\alpha} \int_{\sigma_{\alpha,k}} \chi_\alpha F \dif \mu_\alpha(k).$$
Let $\dif S_{\alpha,k}$ be the surface measure on $\sigma_{\alpha,k}$.
Then
$$\int_M \chi_\alpha \star |T_\lambda| = \int_{K_\alpha} \int_{\sigma_{\alpha,k}} \chi_\alpha \dif S_{\alpha,k} \dif \mu_\alpha(k),$$
so summing in $\alpha$, we obtain 
\begin{equation}\label{calibration condition contr}
\sum_\alpha \int_{K_\alpha} \int_{\sigma_{\alpha,k}} \chi_\alpha F \dif \mu_\alpha(k) = \Mass(\lambda) = \sum_\alpha \int_{K_\alpha} \int_{\sigma_{\alpha,k}} \chi_\alpha \dif S_{\alpha,k} \dif \mu_\alpha(k).
\end{equation}

We claim that $\lambda$ is \dfn{almost calibrated} in the sense that for every $\alpha$ and $\mu_\alpha$-almost every $k$, $\sigma_{\alpha, k}$ is calibrated.
If this is not true, then we may select $\beta$ and $K \subseteq K_\beta$ with $\mu_\beta(K) > 0$, such that for every $k \in K$, $\int_{\sigma_{\beta, k}} F < \vol(\sigma_{\beta, k})$.
Since $0 < \chi_\beta \leq 1$ and $F/\dif S_{\beta, k} \leq 1$ on $\sigma_{\beta, k}$, this is only possible if 
$$\int_{\sigma_{\beta, k}} \chi_\beta F < \int_{\sigma_{\beta, k}} \chi_\beta \dif S_{\beta, k}.$$
Integrating over $K$, and using the fact that in general we have $\int_{\sigma_{\alpha, k}} \chi_\alpha F \leq \int_{\sigma_{\alpha, k}} \chi_\alpha \dif S_{\alpha, k}$, we conclude that 
$$\sum_\alpha \int_{K_\alpha} \int_{\sigma_{\alpha, k}} \chi_\alpha F \dif \mu_\alpha(k) < \sum_\alpha \int_{K_\alpha} \int_{\sigma_{\alpha, k}} \chi_\alpha \dif S_{\alpha, k} \dif \mu_\alpha(k)$$
which contradicts (\ref{calibration condition contr}).

To upgrade $\lambda$ from an almost calibrated lamination to a calibrated lamination, we first, given $\sigma_{\alpha, k}$, choose $k_j$ such that $\sigma_{\alpha, k_j}$ is calibrated and $k_j \to k$.
By Proposition \ref{Hodge theorem}, we can find a continuous $d - 2$-form $A$ defined near $\sigma_{\alpha, k}$ with $F = \dif A$.
This justifies the following application of Stokes' theorem: 
$$\int_{\sigma_{\alpha, k}} F = \int_{\partial \sigma_{\alpha, k}} A.$$
Since $k_j \to k$, and $A$ is continuous,
\begin{align*}
\Mass(\sigma_{\alpha, k}) &= \lim_{j \to \infty} \Mass(\sigma_{\alpha, k_j}) = \lim_{j \to \infty} \int_{\sigma_{\alpha, k_j}} F = \lim_{j \to \infty} \int_{\partial \sigma_{\alpha, k_j}} A = \int_{\partial \sigma_{\alpha, k}} A = \int_{\sigma_{\alpha, k}} F.
\end{align*}

To establish the converse, suppose that $\lambda$ is $F$-calibrated, and let notation be as above.
Since $\lambda$ is $F$-calibrated, for every $\alpha$ and every $k$, the area form on $\sigma_{\alpha, k}$ is $F$. Therefore
\begin{align*}
\int_M T_\lambda \wedge F &= \sum_\alpha \int_{K_\alpha} \int_{\sigma_{\alpha, k}} \chi_\alpha F \dif \mu_\alpha(k) = \Mass(T_\lambda). \qedhere
\end{align*}
\end{proof}

Suppose that $M$ is closed.
A current $T$ is \dfn{homologically minimizing} if it minimizes its mass in its cohomology class.
In particular, if $T$ is calibrated, then $T$ is homologically minimizing.
A measured oriented lamination $\lambda$ is \dfn{homologically minimizing} if its Ruelle-Sullivan current $T_\lambda$ is.

\begin{proposition}\label{properties of calibrated laminations}
Suppose that $M$ is closed.
Let $F$ be a calibration, and let $\lambda$ be a measured oriented $F$-calibrated lamination.
Then:
\begin{enumerate}
\item $\lambda$ is minimal and homologically minimizing.
\item If $G$ is a calibration and cohomologous to $F$, then $\lambda$ is $G$-calibrated.
\end{enumerate}
\end{proposition}
\begin{proof}
Every leaf of $\lambda$ is $F$-calibrated, hence minimal, so $\lambda$ is also minimal.
Since $\lambda$ is $F$-calibrated, so is $T_\lambda$ by Proposition \ref{calibration condition}, but then by (\ref{Ruelle Sullivan homology}), it follows that $T_\lambda$ is $G$-calibrated, and hence $\lambda$ is $G$-calibrated.
Moreover, since $T_\lambda$ is $F$-calibrated, a calibration argument shows that $\lambda$ is homologically minimizing.
\end{proof}

Let $F, G$ be cohomologous calibrations on a closed manifold $M$, and let $\lambda$ be a $F$-calibrated lamination.
In general, there is no reason to believe that $\lambda$ is $G$-calibrated, since there could exist a leaf $N$ which accumulates on a closed leaf countably many times.
Then $N$ would be minimal, but $N$ does not have a well-defined homology class.
However, if $\lambda$ is measured oriented, then by Proposition \ref{properties of calibrated laminations}, the following definition is independent of the choice of $F$.

\begin{definition}
Suppose that $M$ is closed and $d \leq 7$.
Let $\rho \in H^{d - 1}(M, \RR)$ be nonzero, and let $F$ be a tight representative of $\rho$.
We say that a measured oriented lamination $\kappa$ is \dfn{dual} to $\rho$ if $\kappa$ is $F/\|F\|_{L^\infty}$-calibrated.
\end{definition}

We are now ready to prove Corollary \ref{existence of dual laminations}:

\begin{corollary}\label{existence of dual laminations proof}
Suppose that $M$ is closed and $d \leq 7$.
Let $\rho \in H^{d - 1}(M, \RR)$ be nonzero.
Then there exists a measured oriented lamination $\kappa$ which is dual to $\rho$.
\end{corollary}
\begin{proof}
By Theorem \ref{existence 1}, there exists a tight representative $F$ of $\rho$ and a $1$-harmonic conjugate $u$ of $F$.
Let $G := F/\|F\|_{L^\infty}$, and observe that $G$ is a calibration and $u$ is a $1$-harmonic conjugate of $G$.
By Theorem \ref{1 harmonic is MOML}, there exists a measured oriented lamination $\kappa$ such that $\dif u = T_\kappa$.
Therefore 
$$\Mass(\kappa) = \int_M \star |T_\kappa| = \int_M \star |\dif u| = \int_M \dif u \wedge G = \int_M T_\kappa \wedge G.$$
So $T_\kappa$ is $G$-calibrated, and the result follows from Proposition \ref{calibration condition}.
\end{proof}

%%%%%%%%%%%%%%%%%%%
\subsection{Proof of Theorem \ref{lams are calibrated}}\label{proof of Theorem B}
Let $M$ be a closed oriented Riemannian manifold of dimension $2 \leq d \leq 7$, and let $\rho \in H^{d - 1}(M, \RR)$ be a nonzero class.
To ease notation, we normalize
$$\Comass(\rho) = 1.$$
Let $\kappa$ be a measured oriented lamination, and choose a tight form $F$ to represent $\rho$.

\begin{proposition}
Let $\lambda$ range over measured oriented laminations.
If $\kappa$ is dual to $\rho$, then
	\begin{equation}\label{L equals K formula}
	\sup_\lambda \frac{\langle \rho, [\lambda]\rangle}{\Mass(\lambda)} = \frac{\langle \rho, [\kappa]\rangle}{\Mass(\kappa)} = 1.
	\end{equation}
\end{proposition}
\begin{proof}
Let
$$K :=  \sup_\lambda \frac{\langle \rho, [\lambda]\rangle}{\Mass(\lambda)}.$$
Since $\kappa$ is $F$-calibrated, we obtain from (\ref{Ruelle Sullivan homology}) that
$$\langle \rho, [\kappa]\rangle = \int_M T_\kappa \wedge F = \Mass(\kappa).$$
So it is enough to show that $K \leq 1$.

Let $\lambda$ be a measured oriented lamination.
We can write $T_\lambda = \dif v$ for some $v \in BV(M, \RR)$; then, by (\ref{Ruelle Sullivan homology}) and (\ref{Anzellotti Holder inequality}),
$$\langle \rho, [\lambda]\rangle = \int_M T_\lambda \wedge F = \int_M \dif v \wedge F \leq \Mass(\dif v) = \Mass(\lambda).$$
Taking the supremum in $\lambda$, we deduce $K \leq 1$.
\end{proof}

\begin{proposition}\label{calibrated means measured stretch}
Suppose that
$$\langle \rho, [\kappa]\rangle = \Mass(\kappa).$$
Then there exists a $\pi_1(M)$-equivariant function $u$ of least gradient, such that $T_\kappa = \dif u$ and
\begin{equation}\label{duality for least gradient}
\langle \rho, \PD([\dif u])\rangle = \Mass(\dif u).
\end{equation}
\end{proposition}
\begin{proof}
We write $T_\kappa = \dif u$ for some $u \in BV(\tilde M, \RR)$, so that, by (\ref{Ruelle Sullivan homology}) and the fact that $\kappa$ attains the supremum,
$$\int_M \star |\dif u| = \Mass(\kappa) = \langle \rho, [\kappa]\rangle = \int_M T_\kappa \wedge F = \int_M \dif u \wedge F.$$
By Lemma \ref{measurable function is 1} applied with $\mu(E) := \int_E \star |\dif u|$ for every Borel set $E$, and
$$f := \star \frac{\dif u \wedge F}{|\dif u| \|F\|_{L^\infty}},$$
we see that that (\ref{1 extremality}) holds, so that $u$ is a $1$-harmonic conjugate of $F$.
By Lemma \ref{1 extremality implies least gradient}, $u$ has least gradient.
The duality (\ref{duality for least gradient}) follows from the facts that $\dif u$ and $\kappa$ are Poincar\'e dual and $\Mass(\dif u) = \Mass(\kappa)$.
\end{proof}

\begin{proposition}\label{least gradient implies dual}
Suppose that there exists a $\pi_1(M)$-equivariant function $u$ of least gradient, such that $T_\kappa = \dif u$ and (\ref{duality for least gradient}) holds.
Then $\kappa$ is dual to $\rho$.
\end{proposition}
\begin{proof}
We compute 
$$\int_M T_\kappa \wedge F = \int_M \dif u \wedge F = \langle \rho, \PD([\dif u])\rangle = \Mass(\dif u) = \Mass(T_\kappa)$$
so by Proposition \ref{calibration condition}, $\kappa$ is dual to $\rho$.
\end{proof}

%%%%%%%%%%%%%%%%%%%%%%%%%%%
\subsection{More on dual laminations}\label{dual lamination sec}
We are now ready to prove Corollary \ref{every minimizer is dual}.

\begin{corollary}
Let $M$ be a closed oriented Riemannian manifold of dimension $2 \leq d \leq 7$.
Let $\kappa$ be a homologically minimizing measured oriented lamination in $M$.
Then there exists a nonzero cohomology class $\rho \in H^{d - 1}(M, \RR)$ such that $\kappa$ is dual to $\rho$.
\end{corollary}
\begin{proof}
By the Hanh-Banach theorem, there exists a class $\rho$ such that $\Comass(\rho) = 1$ and
$$\langle \rho, [\kappa]\rangle = \Mass([\kappa]).$$
Let $F$ be a tight representative of $\rho$.
Then $\kappa$ attains the supremum in (\ref{duality between stable and comass}), so by Theorem \ref{lams are calibrated}, $\kappa$ is dual to $\rho$.
\end{proof}

For every $\rho$, there are $q$-harmonic functions $u_q$ converging to a least gradient function $u$, such that $\dif u$ is Ruelle-Sullivan for a dual lamination $\lambda$ to $\rho$.
We now show that not every dual lamination arises in this manner.

\begin{example}
We work on the flat square torus $\mathbf T^2 = \Sph^1_x \times \Sph^1_y$.
By Lemma \ref{closed sets are Radon supports}, for each $K \subseteq \Sph^1$, there is a Radon measure $\mu_K$ on $\Sph^1$ with support $K$.
Then $\mu_K$ is transverse to the lamination $\lambda_K$ with leaves $\{k\} \times \Sph^1$, where $k$ ranges over $K$.
Any such lamination is dual to the tight $1$-form $\dif y$.

Recall that we have an identification $H^1(M, \RR) \cong \RR^2$.
If $\alpha \in H^1(M, \RR)$ and $1 < q < \infty$ then the $\alpha$-equivariant $q$-harmonic function on $\RR^2$ is the linear function $z \mapsto \alpha \cdot z$.
So if we have classes $\alpha_q \in H^1(M, \RR)$, and $\alpha_q$-equivariant $q$-harmonic functions $u_q$, and we take $q \to 1$, we get a linear function of least gradient.
Such a function corresponds to a geodesic foliation of $\mathbf T^2$, so cannot equal $\lambda_K$ unless $K = \Sph^1$.
\end{example}

\section{The Euler-Lagrange equation for tight forms}\label{infinityMax}
\subsection{Derivation of the PDE}
We would like to imitate Aronsson's derivation of the $\infty$-Laplacian \cite{Aronsson67}, which we now recall.
Let $1 < p < \infty$ and let $u$ be a $p$-harmonic scalar field.
We write the $p$-Laplacian
$$\dif^*(|\dif u|^{p - 2} \dif u) = 0$$
in nondivergence form 
$$(p - 2) \langle \nabla^2 u, \dif u \otimes \dif u\rangle + |\dif u|^2 \Delta u = 0,$$
divide through by $p - 2$, and observe that, if $u$ converges in a suitably strong sense as $p \to \infty$, then we can neglect the second term to recover the $\infty$-Laplacian 
$$\langle \nabla^2 u, \dif u \otimes \dif u\rangle = 0.$$

Katzourakis observed that Aronsson's scheme does not quite work for systems of PDE because the division by $p - 2$ suppresses the part of the second term which is orthogonal to the first term \cite{Katzourakis12}.
In our work, the part of the second term which is orthogonal to the first term, which Katzourakis calls the \dfn{normal part of the full Aronsson equation} and we simply call \dfn{Katzourakis' correction}, implies that tight forms satisfy a certain involutivity condition.

In order to compute the Katzourakis correction, we need some algebraic preliminaries.
Fix a finite-dimensional vector space $V$, and let $\Lambda^k$ denote the $k$th exterior power of $V$.
If $X \in \Lambda^k$, we introduce the space 
$$X \wedge \Lambda^\ell := \{Z \in \Lambda^{k + \ell}: \exists Y \in \Lambda^\ell ~Z = X \wedge Y\}.$$

% \begin{lemma}\label{exterior splitting lemma}
% Let $X \in \Lambda^1$ and $v \in \Lambda^k$.
% Assume that $X \wedge v = 0$, $X \neq 0$, and $k \geq 1$.
% Then $v \in X \wedge \Lambda^{k - 1}$.
% \end{lemma}
% \begin{proof}
% Let $[\ell]^m$ be the set of increasing multiindices of length $m$ taking values in $\{1, \dots, \ell\}$.
% Let $n := \dim V$.
% Since $X$ is nonzero, we can find a basis $\{e^j: 1 \leq j \leq n\}$ of $V$ such that $X = e^n$.
% Given $J := (j_1, \dots, j_m) \in [n]^m$, we write $e^J := \bigwedge_{i=1}^m e^{j_i}$.
% Then $\{e^I: I \in [n]^m\}$ is a basis of $\Lambda^m$, so there are coefficients $c_J \in [n]^k$ such that
% $$v = \sum_{J \in [n]^k} c_J e^J = \sum_{i=1}^{n - 1} \sum_{I \in [n - 1]^{k - 1}} c_{In} e^{In} + \sum_{I \in [n - 1]^{k - 1}} c_{Ii} e^{Ii}.$$
% Therefore 
% $$0 = v \wedge X = \sum_{i=1}^{n - 1} \sum_{I \in [n - 1]^{k - 1}} c_{Ii} e^{Iin} + \sum_{I \in [n - 1]^{k - 1}} c_{In} e^{Inn} = \sum_{J \in [n - 1]^k} c_J e^{Jn}.$$
% Now $\{e^{Jn}: J \in [n-1]^k\}$ is a subset of the basis $\{e^I: I \in [n]^{k + 1}\}$ of $\Lambda^{k + 1}$, and therefore is linearly independent.
% So $c_J = 0$ for every $J \in [n - 1]^k$; in other words, if $J \in [n]^k$ and $c_J \neq 0$, then there exists $I \in [n - 1]^{k - 1}$ such that $J = (I, n)$.
% Now we set
% $$\varphi := -\sum_{I \in [n - 1]^{k - 1}} c_{In} e^I,$$
% so that 
% \begin{align*} 
% X \wedge \varphi &= \sum_{I \in [n - 1]^{k - 1}} c_{In} e^{In} = \sum_{J \in [n]^k} c_J e^J = v. \qedhere 
% \end{align*}
% \end{proof}

Suppose now that $V$ is equipped with an inner product.
Then the Hodge star $\star$ is defined, and we can identify a form $F \in \Lambda^k$ with the linear map % https://q.uiver.app/#q=WzAsMixbMCwwLCJcXExhbWJkYV57ayAtIDF9Il0sWzEsMCwiViJdLFswLDEsIlxcUHNpX0YiXV0=
\[\begin{tikzcd}
	{\Lambda^{k - 1}} & V
	\arrow["{\Psi_F}", from=1-1, to=1-2]
\end{tikzcd}\]
defined by 
$$\langle \Psi_F(X), v\rangle = \langle F, X \wedge v\rangle.$$

\begin{lemma}\label{computing kernel}
For every $F \in \Lambda^k$,
$$\ker \Psi_F = (\star F) \wedge \Lambda^{k - 2}.$$
\end{lemma}
\begin{proof}
We compute
\begin{align*} 
\ker \Psi_F
&= \{X \in \Lambda^{k - 1}: \forall v \in V ~\langle F, X \wedge v\rangle = 0\} \\
&= \{X \in \Lambda^{k - 1}: \forall v \in V ~(\star F) \wedge X \wedge v = 0\} \\ 
&= \{X \in \Lambda^{k - 1}: (\star F) \wedge X = 0\} \\
&= (\star F) \wedge \Lambda^{k - 2}. \qedhere
\end{align*}
\end{proof}

For $v \in V$ and $F \in \Lambda^k$, let $\iota(v, F) = \iota_v F$ be the interior product. 
It interacts with the Hodge star as 
\begin{equation}\label{Hodge and interior}
\star(F \wedge v) = \iota(v, \star F).
\end{equation}

With the above algebra in place, we are now ready to write the PDE for a $p$-tight form in nondivergence form.
Katzourakis' correction is the second equation, (\ref{p tight nondivergence Katzourakis}).
In the third equation (\ref{p tight nondivergence tangential}), we write $\nabla F$ for the covariant derivative of $F$, which is a section of $\Omega^{d - 1} \otimes \Omega^1$ (so that $\langle \nabla F, F\rangle$ is a $1$-form).

\begin{lemma}\label{p tight nondivergence lemma}
Suppose that $2 < p < \infty$, and let $F$ be a $C^1$ $p$-tight form.
Then 
\begin{subequations} \label{p tight nondivergence}
\begin{empheq}[left=\empheqlbrace]{align}
&\dif F = 0 \label{p tight nondivergence closed}\\
&(\star F) \wedge \dif(\star F) = 0 \label{p tight nondivergence Katzourakis}\\
&\langle \nabla F, F\rangle \wedge \star F + \frac{|F|^2}{p - 2} \dif \star F = 0. \label{p tight nondivergence tangential}
\end{empheq}
\end{subequations}
Furthermore, one can rewrite (\ref{p tight nondivergence tangential}) as 
\begin{equation}\label{p tight nondivergence star}
	\iota(\langle \nabla F, F\rangle, F) + (-1)^{d - 1} \frac{|F|^2}{p - 2} \dif^* F = 0.
\end{equation}
\end{lemma}
\begin{proof}
The first equation (\ref{p tight nondivergence closed}) is clear.
To compute the other two, we expand out 
\begin{align*}
0
&= \dif(|F|^{p - 2} \star F) \\
&= (p - 2) |F|^{p - 4} \langle \nabla F, F\rangle \wedge \star F + |F|^{p - 2} \dif \star F.
\end{align*}
If $F = 0$ then (\ref{p tight nondivergence tangential}) is obvious; otherwise, we can divide through by $(p - 2)|F|^{p - 4}$ to conclude (\ref{p tight nondivergence tangential}).
Taking the Hodge star of (\ref{p tight nondivergence tangential}) and using (\ref{Hodge and interior}), we obtain (\ref{p tight nondivergence star}).

Let $\Pi$ be the projection to $\ker \Psi_F$, so that for any $d - 2$-form $X$,
\begin{align*} 
0
&= \langle \Psi_F(\Pi(G)), \langle \nabla F, F\rangle \rangle \\
&= (-1)^d \langle F, \langle \nabla F, F\rangle \wedge \Pi(G)\rangle \\
&= \langle \iota(\langle \nabla F, F\rangle, F), \Pi(G)\rangle.
\end{align*}
So if we multiply both sides of (\ref{p tight nondivergence star}) on the right by $(p - 2) \Pi$, we annihilate the first term and are left with 
$$|F|^2 \dif^* F \circ \Pi = 0.$$
By Lemma \ref{computing kernel}, $\Pi$ is the projection onto $(\star F) \wedge \Omega^{d - 2}$.
Now let $Y$ be a $d - 3$-form, and observe that $\Pi((\star F) \wedge Y) = (\star F) \wedge Y$.
Therefore 
\begin{align*}
0 
&= |F|^2 \langle \dif^* F \circ \Pi, (\star F) \wedge Y\rangle \\
&= |F|^2 \langle \dif^* F, \Pi((\star F) \wedge Y)\rangle \\
&= |F|^2 \langle \dif^* F, (\star F) \wedge Y\rangle \\
&= |F|^2 \langle \iota(\star F, \dif^* F), Y\rangle.
\end{align*}
We claim that this implies 
\begin{equation}\label{almost Katzourakis}
\iota(\star F, \dif^* F) = 0.
\end{equation}
This is clear if $F = 0$ so assume otherwise.
Then we can divide by $|F|^2$. 
Since $Y$ was arbitary, we conclude (\ref{almost Katzourakis}).
From (\ref{almost Katzourakis}) and (\ref{Hodge and interior}) we compute 
\begin{align*} 
0
&= \iota(\star F, \star \dif \star F) \\
&= \star(\dif (\star F) \wedge \star F)
\end{align*}
which yields (\ref{p tight nondivergence Katzourakis}).
\end{proof}

\begin{proposition}
Assume that $F_p$ is a $p$-tight form for each $2 < p < \infty$, and that $F_p \to F$ in $C^1$.
Then the tight form $F$ satisfies 
\begin{subequations} \label{tight nondivergence}
\begin{empheq}[left=\empheqlbrace]{align}
&\dif F = 0 \label{tight nondivergence closed}\\
&(\star F) \wedge \dif(\star F) = 0 \label{tight nondivergence Katzourakis}\\
&\langle \nabla F, F\rangle \wedge \star F = 0. \label{tight nondivergence tangential}
\end{empheq}
\end{subequations}
Moreover, (\ref{tight nondivergence tangential}) can be rewritten 
\begin{equation}\label{tight nondivergence star}
\nabla_\beta F_{\gamma_1 \cdots \gamma_{d - 1}} F^{\gamma_1 \cdots \gamma_{d - 1}} {F^\beta}_{\alpha_1 \cdots \alpha_{d - 2}} = 0.
\end{equation}
\end{proposition}
\begin{proof}
This follows from taking the limit $p \to \infty$ in Lemma \ref{p tight nondivergence lemma}.
\end{proof}

When $d = 2$, we recover the $\infty$-Laplacian by writing $F = \dif v$ and computing 
$$0 = \langle F, \langle \nabla F, F\rangle\rangle = \langle \nabla^2 v, \dif v \otimes \dif v\rangle.$$
(The Katzourakis correction (\ref{tight nondivergence Katzourakis}) vanishes when $d = 2$, by dimension counting.)
Since the $\infty$-Laplacian has a good theory of viscosity solutions, if $d = 2$, then the above derivation is valid even when $F_p$ is only converging in $C^0$ to $F$.
In particular, a tight $1$-form is nothing more than the derivative of an $\infty$-harmonic function, a fact that we have already tacitly used.

\begin{warning}
Though we derived (\ref{tight nondivergence}) as a formal limit of the PDE (\ref{p tight nondivergence}) for a $p$-tight form, if $d \geq 3$ then it does not follow that solutions of (\ref{tight nondivergence}) are themselves tight (in the sense that they are limits of $p$-tight forms).
Thus it is not true \emph{a priori} that solutions of (\ref{tight nondivergence}) have a conjugate $1$-harmonic function.
\end{warning}

%%%%%%%%%%%%%%%%%%%%%%%%%
\subsection{Geometric interpretation and absolute minimizers}\label{EL interpretation}
Throughout this section, let $M$ be a compact oriented Riemannian manifold of dimension $d$, possibly with boundary.

By a \dfn{streamline} of a $C^2$ function $v$, we mean an integral curve of $\nabla v$.
If $v$ is a sufficiently regular $\infty$-harmonic function and $\gamma$ is a streamline of $v$, then $v$ is an affine function on $\gamma$ \cite[\S3]{Aronsson67}.
Furthermore, if $v$ attains its Lipschitz constant on a point of the streamline $\gamma$, then $\gamma$ is a geodesic \cite[Theorem 5.2]{daskalopoulos2020transverse}.
We now show the analogous result for solutions of (\ref{tight nondivergence}).

To this end, first observe that (\ref{tight nondivergence Katzourakis}) is an involutivity condition: it asserts that on the open set $\{|F| > 0\}$, the bundle $\ker(\star F)$ is involutive.
By Frobenius' theorem, there is a foliation of $\{|F| > 0\}$ by integral hypersurfaces of $\ker(\star F)$.
To emphasize the relationship with streamlines, we call a connected integral hypersurface of $\ker(\star F)$ a \dfn{stream manifold} of $F$.

\begin{proposition}
Let $F$ be a $C^1$ solution of (\ref{tight nondivergence}), and let $N$ be a stream manifold of $F$.
Then $|F|$ is constant on $N$, and $F/|F|$ is the area form of $N$.
\end{proposition}
\begin{proof}
We work in coordinates $(x^\alpha)$ such that $N = \{x^d = 0\}$ and $\partial_d$ is the unit normal vector field of $N$.
The assumption on $N$ means that (possibly after reversing orientation), for some scalar field $\psi > 0$ defined on $N$, $\star F|_N = \psi \dif x^d$, and $F/\psi$ is the area form on $N$.
Since $F$ solves (\ref{tight nondivergence tangential}),
\begin{align*}
0 
&= \partial_\alpha(|F|^2)(\star F)_\beta - \partial_\beta(|F|^2)(\star F)_\alpha \\
&= 2\psi(\delta_{\beta d} \partial_\alpha \psi - \delta_{\alpha d} \partial_\beta \psi).
\end{align*}
Taking $\alpha \in \{1, \dots, d - 1\}$ and $\beta = d$ we deduce that $\partial_\alpha \psi = 0$.
Since $\partial_\alpha$ is an arbitrary tangent vector to $N$, we conclude that $\psi$ is constant.
\end{proof}

\begin{corollary}\label{infty Max calibrates}
Let $F$ be a $C^1$ solution of (\ref{tight nondivergence}), let $N$ be a stream manifold of $F$, and suppose that for some $x \in N$, $|F|$ attains a local maximum at $x$.
If $|F(x)| = 1$, then $N$ is $F$-calibrated.
\end{corollary}

A key feature of the optimal Lipschitz extension problem is the Kirszbraun--Valentine theorem \cite{Lang1997}.
This theorem implies that if $u$ is an $\infty$-harmonic map from a euclidean domain $M$ to $\RR^D$, then for every open $U \subseteq M$,
$$\Lip(u, U) = \Lip(u, \partial U).$$
We now show that the analogue of this theorem, where the Lipschitz constant is replaced by the comass, holds for solutions of (\ref{tight nondivergence}).

\begin{theorem}\label{ABC inequality}
Let $F$ be a $C^1$ solution of (\ref{tight nondivergence}), and let $U$ be an open set such that $H_{d - 1}(U, \RR) = 0$.
Then 
$$\|F\|_{C^0(U)} = \|F\|_{C^0(\partial U)},$$
where the right-hand side refers to the full trace, not just the pullback, of $F$.
\end{theorem}
\begin{proof}
It suffices to show that, if $x \in U$, then $|F(x)| \leq \|F\|_{C^0(\partial U)}$.
Thus we may assume that $|F|$ attains a local maximum at $x$, and that $F(x)$ is nonzero.
Let $N$ be the stream manifold of $F$ at $x$.
Since $|F|$ attains a local maximum at $x$, by Corollary \ref{infty Max calibrates}, $N$ is area-minimizing among all hypersurfaces homologous to $N$ relative to $\partial N$, and $|F|$ is constant along $N$.
Since $F(x)$ is nonzero and $F$ is continuous, $N$ avoids a neighborhood of $\{F = 0\}$.

Since $H_{d - 1}(U, \RR) = 0$, $N$ is area-minimizing among all hypersurfaces with the same boundary, and since $N$ is nonempty it follows that $\partial N$ is nonempty.
Since $N$ avoids a neighborhood of $\{F = 0\}$ and the stream manifolds are a foliation of $U \setminus \{F = 0\}$, $\partial N$ must be a nonempty subset of $\partial U$.
Therefore there is a path from $\partial U$ to $x$ along which $|F|$ is constant and $|F(x)| \leq \|F\|_{C^0(\partial U)}$.
\end{proof}

Another interpretation of the proof of Theorem \ref{ABC inequality} is that for every open set $U$ such that $H_{d - 1}(U, \RR) = 0$, $F|_U$ either has a $1$-harmonic conjugate $u \in BV_\loc(U, \RR)$, or $|F|$ has no local maxima in $U$.
Indeed, if $N$ is the stream manifold through a local maximum of $|F|$, then $N = \partial V$ for some open $V \subseteq U$ by the homological assumption.
It follows from the proof that $1_V$ is a $1$-harmonic conjugate of $F|_U$.

For a basis of open sets $\mathscr U$, a form $F$ has \dfn{absolutely minimal comass} on $\mathscr U$ if for every $U \in \mathscr U$ and every $G \in C^0_\cpt(U, \Omega^{d - 1}_{\rm cl})$,
$$\|F\|_{L^\infty(U)} \leq \|F + G\|_{L^\infty(U)}.$$
This definition is an adaptation of the definition of absolutely minimizing Lipschitz functions \cite{Crandall2008}.
We are now ready to prove Theorem \ref{tight are absolute minimizers}:

\begin{corollary}\label{tight and integrable implies infinity maxwell}
Let $F$ be a $C^1$ solution of (\ref{tight nondivergence}).
Then $F$ has absolutely minimal comass on open sets $U$ with $H_{d - 1}(U, \RR) = 0$.
\end{corollary}
\begin{proof}
Let $G \in C^0_\cpt(U, \Omega^{d - 1})$.
Then, by Theorem \ref{ABC inequality},
\begin{align*}
\|F + G\|_{C^0(U)} &\geq \|F + G\|_{C^0(\partial U)} = \|F\|_{C^0(\partial U)} = \|F\|_{C^0(\partial U)}. \qedhere 
\end{align*}
\end{proof}

Corollary \ref{tight and integrable implies infinity maxwell} is a somewhat surprising result.
Another $\infty$-Laplacian system studied by Katzourakis has solutions which are not absolute minimizers \cite{Katzourakis15}.
The difference is that our situation, as in the work of Aronsson \cite{Aronsson67} and Sheffield and Smart \cite{Sheffield12}, there is a generalization of the streamlines of an $\infty$-harmonic function, which are calibrated.

%%%%%%%%%%%%%%%%%
\subsection{Absolute minimizers and the Barron--Jensen--Wang theorem}
As above, let $M$ be a compact oriented Riemannian manifold of dimension $d$, possibly with boundary.
We would like to establish a converse to Corollary \ref{tight and integrable implies infinity maxwell}: namely, that a closed form of absolutely minimal comass solve (\ref{tight nondivergence}).
However, this is not true, as solutions need not satisfy the Katzourakis correction (\ref{tight nondivergence Katzourakis}).
Instead, we show:

\begin{proposition}\label{ABC implies PDE}
Let $F$ be a $C^1$ closed $d - 1$-form.
If $F$ has absolutely minimal comass on small balls, then $F$ solves
\begin{subequations} \label{tight nondivergence noninvolutive}
\begin{empheq}[left=\empheqlbrace]{align}
&\dif F = 0 \label{tight nondivergence noninvolutive closed}\\
&\langle \nabla F, F\rangle \wedge \star F = 0. \label{tight nondivergence noninvolutive tangential}
\end{empheq}
\end{subequations}
\end{proposition}

We will deduce Proposition \ref{ABC implies PDE} as a consequence of a more general theorem of Barron, Jensen, and Wang \cite[Theorem 5.2]{Barron2001}, itself a consequence of Danskin's theorem.
To formulate their theorem in a diffeomorphism-invariant way, let $(E, \nabla)$ be a vector bundle with connection, and introduce its \dfn{first jet bundle}
$$J^1(E) := E \oplus (\Omega^1 \otimes E).$$
We write $(x, \omega, \xi)$ for a typical element of $J^1(E)$, so $x$ is a point of $M$, $\omega$ is a vector in the fiber $E_x$, and $\xi$ is a vector in $\Omega^1_x \otimes E_x$.

Let $f \in C^1(J^1(E), \RR_+)$ be a nonnegative function on $J^1(E)$.
We say that a section $A \in W^{1, \infty}(M, E)$ \dfn{absolutely minimizes} $f$ if, for every open $U \subseteq M$ and $B \in W^{1, \infty}_\cpt(U, E)$,
$$\sup_{x \in U} f(x, A(x), \nabla A(x)) \leq \sup_{x \in U} f(x, A(x) + B(x), \nabla A(x) + \nabla B(x)).$$
We write $f(A, \nabla A)$ to mean the function $x \mapsto f(x, A(x), \nabla A(x))$.

The proof of \cite[Theorem 5.2]{Barron2001} goes through without changes to the diffeomorphism-invarant setting, so we have:

\begin{theorem}[Barron--Jensen--Wang]\label{BJW Theorem}
Suppose that $(E, \nabla)$ is a vector bundle with connection, $f \in C^1(J^1(E), \RR_+)$, and $A \in C^2(M, E)$ absolutely minimizes $f$.
Then 
\begin{equation}
\left\langle \frac{\partial f}{\partial \xi}(A, \nabla A), \dif(f(A, \nabla A))\right\rangle = 0. \label{BJW}
\end{equation}
\end{theorem}

To clarify the types of the terms in (\ref{BJW}), notice that $\partial f/\partial \xi(A, \nabla A)$ is an $E$-valued $1$-form, and $f(A, \nabla A)$ is a function.
So $\dif(f(A, \nabla A))$ is a $1$-form, and the inner product contracts the $1$-forms to make a section of $E$.

We shall be interested in the jet bundle of $\Omega^{d - 2}$.
A vector in $J^1(\Omega^{d - 2})$ defines a $d - 1$-linear function on tangent vectors, so it is meaningful to take its antisymmetrization.
Therefore we have the \dfn{antisymmetrizer}
% https://q.uiver.app/#q=WzAsMixbMCwwLCJKXjEoXFxPbWVnYV57ZCAtIDJ9KSJdLFsyLDAsIlxcT21lZ2Fee2QgLSAxfSJdLFswLDEsIlxceGkgXFxtYXBzdG8gXFx4aV57XFxybSBhc30iXV0=
\[\begin{tikzcd}
	{J^1(\Omega^{d - 2})} && {\Omega^{d - 1}}
	\arrow["{\xi \mapsto \xi^{\rm as}}", from=1-1, to=1-3]
\end{tikzcd}\]

\begin{proof}[Proof of Proposition \ref{ABC implies PDE}]
Since the statement is local, we may (by a suitable modification of Proposition \ref{Hodge theorem}) assume that $F = \dif A$ for some $d - 2$-form $A$, and that $F$ has (globally) absolutely minimal comass.

We define a function $f$ on the jet bundle of $\Omega^{d - 2}$ by 
$$f(x, \omega, \xi) := |\xi^{\rm as}|^2.$$
In particular $\xi \in \Omega^{d - 1}_x$, and for any $B \in W^{1, \infty}(M, \Omega^{d - 2})$,
$$f(B, \nabla B) = |(\nabla B)^{\rm as}|^2 = |\dif B|^2.$$
So if $U \subseteq M$ is open and $B \in W^{1, \infty}(U, \Omega^{d - 2})$, then, since $F$ has absolutely minimal comass,
$$f(A + B, \nabla A + \nabla B) = |(\nabla A + \nabla B)^{\rm as}|^2 = |F + \dif B|^2 \geq |F|^2 = |(\nabla A)^{\rm as}|^2 = f(A, \nabla A).$$
So $A$ absolutely minimizes $f$, so that $A$ solves (\ref{BJW}).

It remains to simplify (\ref{BJW}).
Since the antisymmetric $\xi^{\rm as}$ and symmetric $\xi^{\rm sym}$ parts of the $d - 1$-linear functional $\xi$ are orthogonal,
$$\frac{\partial f}{\partial \xi} = \frac{\partial}{\partial \xi} |\xi^{\rm as}|^2 = \frac{\partial}{\partial \xi^{\rm as}} |\xi^{\rm as}|^2 + \frac{\partial}{\partial \xi^{\rm sym}} |\xi^{\rm as}|^2 = \frac{\partial}{\partial \xi^{\rm as}} |\xi^{\rm as}|^2 = 2\xi^{\rm as}.$$
In particular, $\partial f/\partial \xi(A, \nabla A) = 2F$.
Moreover, 
$$\dif(f(A, \nabla A)) = \dif(|F|^2) = 2\langle \nabla F, F\rangle.$$
Plugging this into (\ref{BJW}), we conclude 
$$\iota(2\langle \nabla F, F\rangle^\sharp, 2F) = 0,$$
Applying (\ref{Hodge and interior}) and taking Hodge stars, we conclude (\ref{tight nondivergence noninvolutive tangential}).
\end{proof}

In our next two examples, we show that there are solutions of (\ref{tight nondivergence noninvolutive}) which are not minimizers, and that there are solutions of (\ref{tight nondivergence noninvolutive}) which are absolute minimizers but do not satisfy the Katzourakis correction (\ref{tight nondivergence Katzourakis}).

\begin{example}\label{integrability needed}
Let $V$ be the unit vertical vector to the Hopf fibration% https://q.uiver.app/#q=WzAsMyxbMCwwLCJcXG1hdGhiZiBTXjEiXSxbMSwwLCJcXG1hdGhiZiBTXjMiXSxbMiwwLCJcXG1hdGhiZiBTXjIiXSxbMCwxLCIiLDAseyJzdHlsZSI6eyJ0YWlsIjp7Im5hbWUiOiJob29rIiwic2lkZSI6InRvcCJ9fX1dLFsxLDIsIiIsMCx7InN0eWxlIjp7ImhlYWQiOnsibmFtZSI6ImVwaSJ9fX1dXQ==
\[\begin{tikzcd}
	{\mathbf S^1} & {\mathbf S^3} & {\mathbf S^2}
	\arrow[hook, from=1-1, to=1-2]
	\arrow[two heads, from=1-2, to=1-3]
\end{tikzcd}.\]
Since $V$ is a unit vector field, $F := \star V^\flat$ satisfies $2\langle \nabla F, F\rangle = 0$.
However, $2V = \nabla \times V$ \cite[\S3]{Peralta_Salas_2023}; taking the divergence of both sides, we get $2 \nabla \cdot V = 0$, so $\dif F = 0$.
Therefore $F$ solves (\ref{tight nondivergence noninvolutive}) and, since $H^2(\Sph^3, \RR) = 0$, $F$ is exact.

Thus $F$ is a solution which is cohomologous to $0$ and nonzero.
It follows that $F$ cannot have best comass, cannot be an absolute minimizer on open sets $U$ with $H_2(U, \RR) = 0$, and cannot be tight (in the sense that $F$ cannot be approximated by $p$-tight forms).
\end{example}

\begin{example}\label{no Katzourakis}
Let $M$ be the upper hemisphere of $\Sph^3$, and let $F$ be the dual $2$-form to the Hopf fibration as before, but restricted to $M$.
Since $|F| = 1$ is constant, and $F$ is now competing with other forms with the same restriction to $\partial M = \Sph^2$, rather than $0$, $F$ is absolutely minimal.
But $\star F$ is a contact $1$-form, so $F$ cannot be a solution of (\ref{tight nondivergence Katzourakis}).
In particular, $F$ cannot be the $C^1$ limit of $p$-tight forms.
\end{example}

% In fact, there are many Riemannian manifolds for which uniqueness of (\ref{infty Max}) fails.
% Let $M$ be a closed oriented manifold of dimension $d$, and let $F$ be a continuous exact $d - 1$-form. 
% If $F$ is positive on every plane of a distribution $\mathscr D$, then by \cite[Proposition 4.1]{bangert_cui_2017} there exists a Riemannian metric on $M$ for which $F$ is a calibration, and $F$ restricts to the area form on every plane of $\mathscr D$.
% Reasoning as in Example \ref{integrability needed} (where $\mathscr D$ was the distribution of horizontal planes for the fibration), we see that $F$ is a solution of (\ref{infty Max}).
% Furthermore, $F$ calibrates every integral hypersurface of $\mathscr D$, but $F$ is exact, so $\mathscr D$ admits no complete integral hypersurfaces.
% So it is natural to look for totally noninvolutive distributions: in other words, contact structures.
% The simplest nontrivial example of a contact structure is the distribution of horizontal planes to the Hopf fibration.

%%%%%%%%%%%%%%%
\subsection{A topological obstruction to the Kirszbraun--Valentine theorem}
We next check that the hypothesis $H_{d - 1}(U, \RR) = 0$ cannot be avoided in Theorem \ref{ABC inequality}.
The counterexample is a map which has already appeared in Daskalopoulos and Uhlenbeck's work on $\infty$-harmonic maps \cite[\S8]{daskalopoulos2022}.

\begin{example}\label{exactness needed}
Let $U$ be the cylinder $\Sph^1_\theta \times (-1, 1)_x$ with the hyperbolic metric
$$g := \dif x^2 + \cosh^2 x \dif \theta^2.$$
Similarly to Example \ref{boundaries bad}, the Christoffel symbol ${\Gamma^\theta}_{\theta \theta}$ vanishes, so
$$\Delta_\infty \theta = \langle \nabla \dif \theta, \dif \theta \otimes \dif \theta\rangle = \langle \nabla \dif \theta, \partial_\theta \otimes \partial_\theta \rangle \sech^4 x =  {\Gamma^\theta}_{\theta \theta} \sech^4 x = 0.$$
Notice that $|\dif \theta|^2 = \sech^2 x$.
This attains its maximum only on the closed geodesic $\Lambda := \{x = 0\}$, so the conclusion of Theorem \ref{ABC inequality} fails with $F := \dif \theta$.
It is not a coincidence that $\Lambda$ is the cycle which generates $H_1(U, \RR)$.
\end{example}

The Kirsbraun--Valentine theorem for optimal Lipschitz maps is known to fail for maps $u: \Hyp^2 \to \Hyp^2$ with $\Lip(u) < 1$ \cite[Example 9.6]{Gu_ritaud_2017}.
Example \ref{exactness needed} exhibits another reason why the Kirszbraun--Valentine theorem may fail: a minimizing Lipschitz map may be forced to stretch a closed cycle more than it stretches the geodesic between any two points on the boundary.

%%%%%%%%%%%%%
\section{Derivation of the dual equations}\label{duality derivation}
We are interested in deriving the duality equation (\ref{strong duality}) and the PDE for a $p$-tight form (\ref{p tight}) from the duality theorem for convex optimization.
This is not strictly necessary for the proofs of the main theorems of this paper, but they probably seem unmotivated without it.
Indeed, Thurston's conjecture prescribes the use of ``the max flow min cut principle, convexity, and $L^0 \leftrightarrow L^\infty$ duality,'' and convex duality is a sort of grand generalization of the max flow min cut principle.

We follow \cite[Chapter IV]{Ekeland99}.
For a reflexive Banach space $X$, we denote by $\hat X$ its dual.
If $I: X \to \RR \cup \{+\infty\}$ is a convex function, we introduce its \dfn{Legendre transform}, the convex function
\begin{align*}
	\hat I: \hat X &\to \RR \cup \{+\infty\}\\
	\xi &\mapsto \sup_{x \in X} \langle \xi, x\rangle - I(x).
\end{align*}
We identify the cokernel of a linear map with the kernel of its adjoint.

\begin{theorem}[Fenchel--Rockafellar duality]\label{abstract convex analysis}
Let $\Lambda : X \to Y$ be a bounded linear map between reflexive Banach spaces.
Let $I: Y \to \RR \cup \{+\infty\}$ satisfy:
\begin{enumerate}
\item $I$ and $\hat I$ are strictly convex,
\item $I$ is lower semicontinuous,
\item $I$ is \dfn{coercive}, in the sense that 
$$\lim_{\|y\|_Y \to \infty} I(y) = +\infty,$$
and
\item there exists a point $x \in X$ such that $I(\Lambda(x)) < +\infty$, and $I$ is continuous at $\Lambda(x)$.
\end{enumerate}
Then:
\begin{enumerate}
\item There exists a minimizer $\underline x \in X$ of $I(\Lambda(x))$, unique modulo $\ker \Lambda$.
\item There exists a unique maximizer $\overline \eta$ of $-\hat I(-\eta)$ subject to the constraint $\eta \in \coker \Lambda$.
\item We have \dfn{strong duality}
\begin{equation}\label{abstract strong duality}
I(\Lambda(\underline x)) = -\hat I(-\overline \eta).
\end{equation}
\end{enumerate}
\end{theorem}
\begin{proof}
This is largely a special case of \cite[Chapter IV, Theorem 4.2]{Ekeland99}.
Let $\mathscr P, \mathscr P^*$ be as in the statement of that theorem.
Then $\mathscr P$ is the problem of minimizing $J(x, \Lambda x)$ where $J(x, y) := I(y)$.
The Legendre transform of $J$ satisfies 
$$\hat J(\xi, \eta) = \begin{cases} \hat I(\eta), & \xi = 0, \\
	+\infty, &\xi \neq 0,
\end{cases}$$
and $\mathscr P^*$ is the problem of maximizing
$$-\hat J(\Lambda^* \eta, -\eta) = \begin{cases}
	-\hat I(-\eta), &\eta \in \ker \Lambda^*, \\
	-\infty, &\eta \notin \ker \Lambda^*,
\end{cases}$$
where $\Lambda^*$ is the adjoint of $\Lambda$.
Then most of the various assertions of this theorem follow immediately from \cite[Chapter IV, Theorem 4.2]{Ekeland99}.
The fact that $\overline \eta \in \coker \Lambda$ follows from the facts that $\overline \eta$ is a solution of $\mathscr P^*$, but any solution of $\mathscr P^*$ must be a member of $\ker \Lambda^*$. 
To establish uniqueness, we use \cite[Chapter II, Proposition 1.2]{Ekeland99}, the fact that $\hat I$ is strictly convex, and the fact that we may view $I \circ \Lambda$ as a strictly convex function on the reflexive Banach space $X/\ker \Lambda$.
\end{proof}

Let $M$ be a closed oriented Riemannian manifold with universal covering $\tilde M$, let $\Gamma := \pi_1(M)$, and let $\alpha \in \Hom(\Gamma, \RR)$.
As always we identify $\alpha$ with a harmonic $1$-form on $M$.
Let $1 < p < \infty$ and $1/p + 1/q = 1$.

\begin{theorem}[convex duality for the $q$-Laplacian]\label{mfmc qLaplacian}
There is an $\alpha$-equivariant function $u$ on the universal cover $\tilde M$, and a closed $d - 1$-form $F$, such that:
\begin{enumerate}
\item $u$ is the unique $\alpha$-equivariant $q$-harmonic function modulo constants.
\item $F$ is the unique $d - 1$-form which minimizes
$$J_\alpha(F) := \frac{1}{p} \int_M \star |F|^p + \int_M \alpha \wedge F$$
subject to the constraint $\dif F = 0$.
\item $F$ is the unique closed $d - 1$-form such that
\begin{equation}\label{strong duality appendix}
	\frac{1}{q} \int_M \star |\dif u|^q + \frac{1}{p} \int_M \star |F|^p + \int_M \dif u \wedge F = 0.
\end{equation}
\end{enumerate}
\end{theorem}
\begin{proof}
Let $\Lambda$ be the map 
% https://q.uiver.app/#q=WzAsMixbMCwwLCJXXnsxLCBxfShNLCBcXG1hdGhiZiBSKSJdLFsyLDAsIkxecShNLCBcXE9tZWdhXjEpIl0sWzAsMSwiXFxtYXRocm0gZCJdXQ==
\[\begin{tikzcd}
	{W^{1, q}(M, \mathbf R)} & {L^q(M, \Omega^1)}
	\arrow["{\mathrm d}", from=1-1, to=1-2]
\end{tikzcd}\]
so $\ker \Lambda$ is the space of constant functions.
We identify the dual space of $W^{1, q}(M, \RR)$ with $W^{-1, p}(M, \Omega^d)$ by integration.
Similarly, we have a bilinear function
\begin{align*}
	L^q(M, \Omega^1) \times L^p(M, \Omega^{d - 1}) &\to \RR \\
	(\psi, F) &\mapsto \int_M \psi \wedge F
\end{align*}
which is easily seen to be a perfect pairing.
Therefore the dual space of $L^q(M, \Omega^1)$ can be identified with $L^p(M, \Omega^{d - 1})$.
For these identifications, an integration by parts shows that the adjoint map $\Lambda^*$ is
% https://q.uiver.app/#q=WzAsMixbMCwwLCJMXnAoTSwgXFxPbWVnYV57ZCAtIDF9KSJdLFsyLDAsIldeey0xLCBwfShNLCBcXE9tZWdhXmQpIl0sWzAsMSwiLVxcbWF0aHJtIGQiXV0=
\[\begin{tikzcd}
	{L^p(M, \Omega^{d - 1})} & {W^{-1, p}(M, \Omega^d)}
	\arrow["{-\mathrm d}", from=1-1, to=1-2]
\end{tikzcd}\]
Thus we have a natural isomorphism
$$\coker \Lambda = L^p(M, \Omega^{d - 1}_{\rm cl}).$$

Next we introduce the function
\begin{align*}
I_\alpha: L^q(M, \Omega^1) &\to \RR \\
\psi &\mapsto \frac{1}{q} \int_M \star |\psi + \alpha|^q.
\end{align*}
Clearly $I_\alpha$ is continuous, strictly convex, finite, and coercive.
As in \cite[Chapter IV, \S2.2]{Ekeland99}, we compute 
$$\widehat{I_\alpha}(F) = \widehat{I_0}(F) - \int_M \alpha \wedge F = \frac{1}{p} \int_M \star |F|^p - \int_M \alpha \wedge F = J_\alpha(-F).$$
In particular, $\widehat{I_\alpha}$ is strictly convex.

We have now checked all the hypotheses of Theorem \ref{abstract convex analysis} and can rattle off conclusions as follows:
\begin{enumerate}
\item There is a minimizer $v \in W^{1, q}(M, \RR)$ of $I_\alpha(\dif v)$. Furthermore, $v$ is unique modulo constants.
We then lift $v$ to a function $\tilde v \in W^{1, q}_\loc(\tilde M, \RR)$ which is invariant under the action of $\Gamma$.
Since $\alpha$ lifts to a harmonic $1$-form $\dif w$ on $\tilde M$, where $w \in C^\infty(\tilde M, \RR)$ is unique modulo constants.
Then we set 
$$u := v + w,$$
which is well-defined up to a constant, and observe that $\dif u = \dif v + \alpha$.
In particular, $u$ is $\alpha$-equivariant and minimizes $I_0(\dif u) = I_\alpha(\dif v)$.
The Euler-Lagrange equation for $I_0 \circ \dif$ is the $q$-Laplacian.
\item There is a unique maximizer $F \in L^p(M, \Omega^{d - 1}_{\rm cl})$ of
$$-J_\alpha(F) = \frac{1}{p} \int_M \star |F|^p + \int_M \alpha \wedge F.$$
It is equivalent that $F$ maximize $-J_\alpha$ and minimize $J_\alpha$.
\item We have (\ref{abstract strong duality}), which reads 
$$0 = I_\alpha(\dif v) + \widehat{I_\alpha}(-F) = I_0(\dif u) + J_\alpha(F) = \frac{1}{q} \int_M |\dif u|^q + \frac{1}{p} \int_M |F|^p + \int_M \alpha \wedge F.$$
Since $F$ is closed, we can replace $\alpha$ with the cohomologous $1$-form $\dif u$ in the last integral to conclude (\ref{strong duality appendix}).
\end{enumerate}

Finally, suppose that $G$ also satisfies (\ref{strong duality appendix}).
Then, replacing $\alpha$ and $\dif u$ as above,
$$J_\alpha(G) = \frac{1}{p} \int_M \star |G|^p + \int_M \alpha \wedge G = -\frac{1}{q} \int_M |\dif u|^q = \frac{1}{p} \int_M \star |F|^p + \int_M \alpha \wedge F = J_\alpha(F).$$
Since $F$ is the unique minimizer of $J_\alpha$, it follows that $G = F$.
\end{proof}

\begin{corollary}
Let $u, F$ be as Theorem \ref{mfmc qLaplacian}.
Then $F$ is $p$-tight, and we have the pointwise duality relation
\begin{equation}\label{dual solution appendix}
F = -|\dif u|^{q - 2} \star \dif u.
\end{equation}
\end{corollary}
\begin{proof}
Suppose that
$$G := -|\dif u|^{q - 2} \star \dif u.$$
Then, we use $(q - 1)p = q$ to compute
$$\frac{1}{q} \int_M \star |\dif u|^q + \frac{1}{p} \int_M \star |G|^p + \int_M \dif u \wedge G = 0,$$
so by the uniqueness results in Theorem \ref{mfmc qLaplacian}, $F = G$.
In particular (\ref{dual solution appendix}) holds.
Finally we use $(q - 1)(p - 2) = 0$ to compute 
$$\dif(|F|^{p - 2} \star F) = \dif(|\dif u|^{(q - 1)(p - 2)} \dif u) = \dif^2 u = 0.$$
Combining this with $\dif F = 0$, we see that $F$ is $p$-tight.
\end{proof}

Finally, let us observe that when $d = 2$, we can locally write $F = \dif v$ for a function $v$.
Then (\ref{dual solution appendix}) is the equation for the conjugate $q$-harmonic that was studied by Daskalopoulos and Uhlenbeck \cite[\S3]{daskalopoulos2020transverse}.

%%%%%%%%%%%%%%%%
\appendix
\section{Measure theory}\label{GMT appendix}
\begin{proposition}[regularized Poincar\'e lemma]\label{Hodge theorem}
Let $M$ be a Riemannian manifold, $x \in M$, and $\ell \in \{1, \dots, d\}$.
Then there exists $r_* > 0$ which depends only on $\Riem_M$ near $x$ and the injectivity radius of $x$, such that for every $0 < r \leq r_*$ and $F \in L^\infty(B(x, r), \Omega^\ell_{\rm cl})$, there exists a H\"older continuous $\ell - 1$-form $A$ such that $F = \dif A$.
\end{proposition}
\begin{proof}
We may choose $r_*$ so that the exponential map $B_{\RR^d}(0, r_*) \to B(x, r_*)$ is a diffeomorphism which induces topological isomorphisms for every function space under consideration.
Thus it is no loss to replace $B(x, r)$ with the unit euclidean ball $\Ball^d$.
By the main theorem of \cite{Costabel2010}, for every $1 < p < \infty$ there is a continuous right inverse to the exterior derivative
% https://q.uiver.app/#q=WzAsMixbMCwwLCJXXnsxLCBwfShVLCBcXE9tZWdhXntcXGVsbCAtIDF9KSJdLFsxLDAsIkxecChVLCBcXE9tZWdhXlxcZWxsX3tcXHJtIGNsfSkiXSxbMCwxLCJcXG1hdGhybSBkIl1d
\[\begin{tikzcd}
	{W^{1, p}(\Ball^d, \Omega^{\ell - 1})} & {L^p(\Ball^d, \Omega^\ell_{\rm cl})}
	\arrow["{\mathrm d}", from=1-1, to=1-2]
\end{tikzcd}.\]
The result now follows from the Sobolev embedding theorem if we take $p > d$.
\end{proof}

\begin{proposition}\label{portmanteau}
Let $(\mu_n)$ be a sequence of positive Radon measures with $\mu_n(X) \lesssim 1$, and $\mu$ a positive Radon measure, on a compact metrizable space $X$.
Then the following are equivalent:
\begin{enumerate}
\item $\mu_n \weakto^* \mu$.
\item $\liminf_{n \to \infty} \mu_n(X) \geq \mu(X)$, and for every closed $Y \subseteq X$, $\limsup_{n \to \infty} \mu_n(Y) \leq \mu(Y)$.
\item $\limsup_{n \to \infty} \mu_n(X) \leq \mu(X)$, and for every open $U \subseteq X$, $\liminf_{n \to \infty} \mu_n(U) \geq \mu(U)$.
\end{enumerate}
\end{proposition}
\begin{proof}
After rescaling so that $\mu_n(X) \leq 1$ for every $n$, this follows from \cite[Theorem 13.16]{klenke2013probability}.
\end{proof}

\begin{lemma}\label{measurable function is 1}
Let $\mu$ be a positive finite measure on a measurable space $X$, and let $f$ be a measurable function with $\|f\|_{L^\infty(X, \mu)} \leq 1$.
If $\int_X f \dif \mu = \mu(X)$, then $f = 1$, $\mu$-almost everywhere.
\end{lemma}
\begin{proof}
Let $E \subseteq X$ be a measurable set.
Then $\int_E f \dif \mu \leq \mu(E)$.
Conversely, since $X \setminus E$ is also measurable,
$$\int_E f \dif \mu = \int_X f \dif \mu - \int_{X \setminus E} f \dif \mu \geq \mu(X) - \mu(X \setminus E) = \mu(E),$$
so $\int_E f \dif \mu = \mu(E)$.
Since $E$ was arbitrary, we conclude $f = 1$, $\mu$-almost everywhere.
\end{proof}

\begin{lemma}\label{closed sets are Radon supports}
Let $K \subseteq \RR$ be a nonempty compact set.
Then there exists a Radon measure $\mu$ on $\RR$ whose support is $K$.
\end{lemma}
\begin{proof}
By the Cantor-Bendixson theorem, we may assume that $K$ either is countable, or has no isolated points.
Clearly such a measure exists if $K$ is countable, so suppose that $K$ has no isolated points.
By removing those connected components of $K$ which are closed intervals, we may assume that $K$ is totally disconnected.
But then $K$ is a totally disconnected nonempty compact metrizable space without isolated points, so $K$ is homeomorphic to the Cantor set, which admits a Radon measure.
\end{proof}

\subsection{Anzellotti's pairing}
We are going to appeal to results of Anzellotti \cite{Anzellotti1983}, which were formulated in the gradient-divergence-curl formalism of vector calculus on euclidean space.
In fact, these results can be formulated in terms of differential forms, after which they become manifestly invariant under bi-Lipschitz changes of coordinates.
To demonstrate the reformulation process, we go through the details of the construction of the Anzellotti wedge product in the language of differential forms, but leave the reformulation of Anzellotti's trace theorem to the reader.

\begin{definition}
Let $u \in BV(M, \Omega^k)$ and $F \in L^\infty(M, \Omega^{d - k - 1})$.
Assume that $\dif F \in L^d(M, \Omega^{d - k})$.
Then the \dfn{Anzellotti wedge product} of $\dif u$ and $F$ is the distribution $\dif u \wedge F$, such that for every test function $\chi \in C^\infty_\cpt(M, \RR)$,
$$\langle \dif u \wedge F, \chi\rangle := -\int_M \chi u \wedge \dif F - \int_M \dif \chi \wedge u \wedge F.$$
\end{definition}

\begin{proposition}\label{Anzellotti wedge product exists}
Let $u \in BV(M, \Omega^k)$, $F \in L^\infty(M, \Omega^{d - k - 1})$, and $\dif F \in L^d(M, \Omega^{d - k})$.
Then the Anzellotti wedge product $\dif u \wedge F$ is well-defined as a distribution.
In fact, $\dif u \wedge F$ is a Radon measure, and 
\begin{equation}\label{Anzellotti Holder inequality}
\Mass(\dif u \wedge F) \leq \Comass(F) \Mass(\dif u).
\end{equation}
\end{proposition}
\begin{proof}
We follow the proof of \cite[Theorem 1.5]{Anzellotti1983}.
By the Sobolev embedding theorem for $BV$ \cite[\S5.6]{evans2015measure}, $u \in L^{\frac{d}{d - 1}}(M, \Omega^k)$, the dual space of $L^d(M, \Omega^{d - k})$.
Therefore for every $\chi \in C^\infty_\cpt(M)$, $\langle \dif u \wedge F, \chi\rangle$ is finite, so $\dif u \wedge F$ is well-defined as a distribution.

Suppose that $\supp \chi \Subset U$ for some $U \Subset M$.
If $u$ is sufficiently smooth, then an integration by parts gives 
$$|\langle \dif u \wedge F, \chi\rangle| = \left|\int_M \chi \dif u \wedge F\right| \leq \Comass(F) \|\chi\|_{C^0} \Mass_U(\dif u)$$
where $\Mass_U$ denotes the mass \emph{computed in $U$}.
In general, we can find a sequence $(u_n) \subset C^\infty$ such that $u_n \weakto^* u$ in $BV$.
Then $u_n \weakto u$ in $L^{\frac{d}{d - 1}}$ and $\dif u_n \weakto^* \dif u$ as currents of locally finite mass.
Since we are testing $\dif u$ against the $L^d$ form $\chi F$,
$$|\langle \dif u \wedge F, \chi\rangle| \leq \liminf_{n \to \infty} |\langle \dif u_n \wedge F, \chi\rangle| \leq \Comass(F) \|\chi\|_{C^0} \liminf_{n \to \infty} \Mass_U(\dif u_n).$$
But, by Proposition \ref{portmanteau},
$$\liminf_{n \to \infty} \Mass_U(\dif u_n) \leq \liminf_{n \to \infty} \Mass_{\overline U}(\dif u_n) \leq \Mass(\dif u)$$
which gives the desired estimate (\ref{Anzellotti Holder inequality}), since we only used the $C^0$ norm of $\chi$.
\end{proof}

\begin{proposition}[trace theorem {\cite[Theorem 1.2]{Anzellotti1983}}]\label{integration is welldefined}
Let $\iota: N \to M$ be the inclusion of an oriented Lipschitz hypersurface.
Let $\mathcal X$ be the space of $F \in L^\infty(M, \Omega^{d - 1})$ such that the components of $\dif F$ are Radon measures.
Then the pullback $\iota^*$ of $d - 1$-forms extends to a bounded linear map
$$\iota^*: \mathcal X \to L^\infty(N, \Omega^{d - 1})$$
satisfying the estimate
\begin{equation}\label{integral over chain is linfinity}
	\|\iota^* F\|_{L^\infty(N)} \leq \|F\|_{L^\infty(M)}.
\end{equation}
\end{proposition}

% A measurable set $U \subseteq M$ has \dfn{locally finite perimeter} if $1_U \in BV_\loc$.
% The \dfn{measure-theoretic boundary} of $U$ is the set $\partial U$ of $x \in M$ such that for all $\varepsilon > 0$,
% $$0 < \Mass(U \cap B(x, \varepsilon)) < \Mass(B(x, \varepsilon)).$$
% Then $\partial U$ is an integral $d - 1$-current (in particular, the sum of Lipschitz hypersurfaces), with surface measure $\star |\dif 1_U|$; moreover, $\dif 1_U$ is conormal to $\partial U$ \cite[Theorem 4.4]{Giusti77}.
% See also \cite[Appendix A]{BackusCML} for a discussion of related issues in our setting; in particular, we assert that without loss of generality, we may assume that $\partial U$ is closed.
% The superlevel sets $\{u > \lambda\}$ of a function $u \in BV_\loc(M, \RR)$ have locally finite perimeter \cite[Theorem 1.23]{Giusti77}.

% \begin{proposition}[coarea formula]\label{coarea theorem}
% Let $u \in BV_\loc(M, \RR)$, $F \in L^\infty_\loc(M, \Omega^{d - 1})$, and $\chi \in C^0_\cpt(M, \RR)$.
% If $\dif F \in L^d_\loc$, then
% \begin{equation}\label{coarea formula}
% \int_M \chi \dif u \wedge F = \int_{-\infty}^\infty \int_{\partial \{u > \lambda\}} \chi F \dif \lambda.
% \end{equation}
% \end{proposition}
% \begin{proof}
% After identifying $\dif 1_U$ with the integral current $\partial U$, this follows from \cite[Proposition 2.7(ii)]{Anzellotti1983}.
% \end{proof}

\printbibliography

\end{document}